\documentclass[a4paper,12pt]{amsart}

\usepackage{amssymb, amscd}
\usepackage{latexsym,epsfig}
\usepackage[all]{xy}
\usepackage{pst-all}
\numberwithin{equation}{section}
\def\today{\ifcase\month\or Jan\or Febr\or  Mar\or  Apr\or May\or Jun\or  Jul\or Aug\or  Sep\or  Oct\or Nov\or  Dec\or\fi \space\number\day, \number\year}

\newcommand{\VVl}{\VV_{\mu}}	% !!

\newcommand{\Sy}{{\mathrm{Sym}}} 
\newcommand{\CC}{\mathbb C}
\newcommand{\EE}{\mathbb E}
\newcommand{\FF}{\mathbb F}

\newcommand{\HH}{\mathbb H}

\newcommand{\NN}{\mathbb N}
\newcommand{\PP}{\mathbb P}
\newcommand{\QQ}{\mathbb Q}

\newcommand{\VV}{\mathbb V}

\newcommand{\ZZ}{\mathbb Z}
\newcommand\M[1]{{\mathcal M}_{#1}}
\newcommand\Mct[1]{{\mathcal M}_{#1}^c}
\newcommand\Hy[1]{{\mathcal H}_{#1}}
\newcommand\barM[1]{\overline{\mathcal M}_{#1}}
\newcommand\barA[1]{\overline{\mathcal A}_{#1}}
\newcommand\Vbar{{\overline{\mathcal V}}_{\mu}}
\newcommand\A[1]{{\mathcal A}_{#1}}
\newcommand\X[1]{{\mathcal X}_{#1}}

\newcommand\SL[2]{{\mathrm{SL}}({#1},{#2})}
\newcommand\GL[2]{{\mathrm{GL}({#1},{#2})}}
\newcommand\PGL[2]{{\mathrm{PGL}({#1},{#2})}}
\newcommand\Sp[2]{{\mathrm{Sp}}({#1},{#2})}

\newcommand\langepijl[1]{\buildrel {#1} \over \longrightarrow}
\newcommand{\Sym}{{\mathrm{Sym}}}

\newcommand{\tensor}{\otimes}

\numberwithin{equation}{section}

\newcommand{\glv}{\mathrm{GL}(V)}

\newcommand{\Vla}{{\VV}_{\mu}^{\prime}}
\newcommand{\VlaC}{{\VV}_{\mu}^{\prime}\otimes{\CC}}
\newcommand{\Vbla}{{\mathcal{V}}_{\mu}}
\newtheorem{theorem}{Theorem}[section]
\newtheorem{lemma}[theorem]{Lemma}
\newtheorem{proposition}[theorem]{Proposition}

\newtheorem{corollary}[theorem]{Corollary}

\newtheorem{definition-lemma}[theorem]{Definition-Lemma}

\theoremstyle{definition}

\newtheorem{example}[theorem]{Example}

\theoremstyle{remark}
\newtheorem{remark}[theorem]{Remark}

\newtheorem{notation}[theorem]{Notation}

\begin{document}

\title[Concomitants of Ternary Quartics and Modular Forms of Genus Three]
{Concomitants of Ternary Quartics and 
Vector-valued Siegel and Teichm\"uller Modular Forms of Genus Three}

\author{Fabien Cl\'ery}
\address{Department of Mathematics,
Loughborough University,
England}
\email{cleryfabien@gmail.com}

\author{Carel Faber}
\address{Mathematisch Instituut, Universiteit Utrecht,
Postbus 80010,
3508 TA Utrecht,
The Netherlands}
\email{C.F.Faber@uu.nl}

\author{Gerard van der Geer}
\address{Korteweg-de Vries Instituut, Universiteit van
Amsterdam, Postbus 94248,
1090 GE  Amsterdam, The Netherlands}
\email{geer@science.uva.nl}

\subjclass{10D, 11F46, 14H10, 14H45, 14J15, 14K10}
\begin{abstract}
We show how one can use the representation theory of ternary quartics to 
construct all vector-valued Siegel modular forms and Teichm\"uller 
modular forms of degree~$3$. 
The relation between the order of vanishing of 
a concomitant on the locus of double conics and the order of vanishing 
of the corresponding modular form on the hyperelliptic locus
plays an important role.
We also determine the connection between
Teichm\"uller cusp forms on $\barM{g}$ and the middle cohomology
of symplectic local systems on $\M{g}\,$.
In genus~$3$, we make this explicit in a large number of cases.
\end{abstract}

\maketitle
%%%%%%%%%%%%%%%%%%%%%%%%%%%%%%%%%%%%%%%%
\centerline{\today}
%%%%%%
\begin{section}{Introduction}
This paper contains two main results. Firstly, we show how the representation
theory associated to ternary quartics can be used to describe and
construct all vector-valued Siegel and Teichm\"uller modular forms
of degree~$3$ (Theorem~\ref{the_isom}).
This uses the classical notion of concomitants, of which invariants,
covariants, and contravariants are special cases.
Secondly, we describe for arbitrary~$g$
the precise relation between certain spaces of
Teichm\"uller cusp forms on $\barM{g}$ and the middle cohomology
of the standard symplectic local systems on $\M{g}$
(Theorem~\ref{topHodge}).
We illustrate the main results by a substantial number of examples,
focusing on genus~$3$ for the second theorem,
and obtain several other results of independent interest.

\bigskip
Let $\A{3}$ be the moduli space of principally polarized abelian varieties
of dimension $3$. 
Over the complex numbers the orbifold $\A{3}({\CC})$,
associated to the moduli space $\A{3}\,$,
can be written as an arithmetic quotient 
$\Gamma_3\backslash \mathfrak{H}_3\,$, where
the Siegel modular group $\Gamma_3=\Sp{6}{\ZZ}$ of degree $3$
acts on the Siegel upper half space $\mathfrak{H}_3$
of degree $3$ in the usual way.
The moduli stack $\A{3}$ carries a natural 
vector bundle ${\EE}$ of rank $3$, the Hodge bundle
with fibre $H^0(X,\Omega_X^1)$ over the point $[X]$ of $\A{3}\,$. 
Over ${\CC}$ it can be given as a quotient 
$\Gamma_3\backslash (\mathfrak{H}_3\times {\CC}^3)$ where the action
on ${\CC}^3$
corresponds to the standard representation of ${\rm GL}(3)$.
In a similar way, for each irreducible representation $\rho$ of 
${\rm GL}(3)$ there is an associated vector bundle ${\EE}_{\rho}$ 
that can be constructed from ${\EE}$ by using a Schur functor.  
Sections of powers of the determinant bundle 
$\det({\EE})$ on $\A{3}$ can be identified with scalar-valued
Siegel modular forms, while sections of ${\EE}_{\rho}$ can be
identified with vector-valued Siegel modular forms.
The vector bundle ${\EE}$, and more generally all ${\EE}_{\rho}$, extend
in a natural way over the standard smooth compactification of $\A{3}$.
By the Koecher principle holomorphic sections of ${\EE}_{\rho}$ extend
over this compactification.

Let $\M{3}$ denote the moduli space of curves of genus $3$. 
The Torelli morphism $t\colon\M{3}\to \A{3}$ 
is a morphism of algebraic stacks of degree $2$ ramified along
the hyperelliptic locus. By pullback under $t$
 we obtain the Hodge bundle ${\EE}^{\prime}$
on $\M{3}$ and for each irreducible representation $\rho$ of ${\rm GL}(3)$
a vector bundle ${\EE}^{\prime}_{\rho}$ on $\M{3}$. Sections of such a bundle
${\EE}^{\prime}_{\rho}$ are called Teichm\"uller modular forms of degree $3$.  
The vector bundle ${\EE}^{\prime}$ and hence all the 
${\EE}^{\prime}_{\rho}$ can be extended in a natural way over the 
Deligne-Mumford compactification 
$\overline{\mathcal M}_3$ and we show that a holomorphic section of 
${\EE}^{\prime}_{\rho}$ automatically extends to a holomorphic section of 
the extended bundle. 

The first Teichm\"uller modular form that is not (a pullback under the
Torelli map of)
a Siegel modular form
is the scalar-valued form $\chi_9$ of weight $9$
whose existence was proven by Ichikawa
\cite{Ichikawa1,Ichikawa2}.
There is an involution on the space of Teichm\"uller forms such that
a Teichm\"uller modular form $F$  that is invariant is the pullback
of a Siegel modular form, 
while an anti-invariant form is divisible by $\chi_9$, with quotient
the pullback of a Siegel modular form. The study of
Teichm\"uller modular forms of degree~$3$ reduces therefore to that
of Siegel modular forms (\S\ref{TMF}).

A nonhyperelliptic curve of genus $3$ has as canonical image a quartic
curve in ${\PP}^2$ and thus the open part ${\mathcal M}_3^{\rm nh}$
of $\M{3}$ that corresponds to nonhyperelliptic curves
has a description as the quotient of an open part of the space
of ternary quartics under the action of~$\mathrm{GL}(3)$.
Thus the representation theory of ternary quartics enters, i.e.,
the decomposition of $\Sym^d(\Sym^4(\CC^3))$ into irreducible
representations of~$\mathrm{GL}(3)$. 
The classical notion of concomitants of ternary quartics
makes this decomposition explicit (see~\cite{Chipalkatti}).
Our first main result 
%(Theorem~\ref{the_isom}) 
is a complete description of Siegel (and Teichm\"uller) modular forms
of degree~$3$ in terms of concomitants.
We associate to a concomitant a meromorphic Siegel or
Teichm\"uller modular form which is holomorphic outside the hyperelliptic locus
on $\M{3}$ or $\A{3}$ (see \S\ref{ICMF}).
This is analogous
to our description of Siegel modular forms of degree~$2$ in terms of 
covariants of the action of $\mathrm{GL}(2)$ on the space of binary sextics,
see \cite{CFvdG}. 

The most basic concomitant is the universal ternary quartic $f$.
It defines a meromorphic Teichm\"uller modular form $\chi_{4,0,-1}$
of weight $(4,0,-1)$. Multiplication by $\chi_9$ makes it into
a holomorphic Siegel modular form $\chi_{4,0,8}\,$,
a section of $\Sym^4({\EE}) \otimes \det^8({\EE})$,
the `first' Siegel cusp form of degree $3$ (cf.~\S\ref{chi408}).

The fact that in general concomitants define meromorphic modular forms
that become holomorphic after multiplication with a suitable power of $\chi_9$
forces us to analyze the order of vanishing of the modular form 
associated to a concomitant along the hyperelliptic locus. 
We express this order of vanishing in terms of the order of vanishing
of the concomitant along the locus of double conics in the space of ternary
quartics. 
Remarkably, this enables us to identify the spaces of concomitants of
ternary quartics with given order of vanishing
along the locus of double
conics with the spaces of Siegel modular forms with 
given order of vanishing
along the divisor at infinity (Theorem~\ref{the_isom}).

Instead of working with $\chi_{4,0,-1}$ and multiplying with a power
of $\chi_9\,$, we can also work with $\chi_{4,0,8}$ and obtain
from a concomitant a holomorphic Siegel modular form, which may be
divisible by a power of $\chi_9\,$.
In order to use this in an efficient way, we need to know the
Fourier expansion of $\chi_{4,0,8}$ rather well.
We obtain it by analyzing the Schottky form,
a scalar-valued Siegel cusp form of degree~$4$ and weight~$8$, 
along the `diagonally embedded'
$\mathfrak{H}_3 \times \mathfrak{H}_1 \subset \mathfrak{H}_4$ in the 
Siegel upper half space of degree~$4$ (see \S\S\ref{chi408}--\ref{FJ}).

To demonstrate our approach, we construct a substantial number
of Siegel cusp forms and we compute some of their Hecke eigenvalues,
finding agreement with~\cite{BFvdG3}.

In \cite{BFvdG2}, Bergstr\"om and two of the present authors studied the
cohomology of symplectic local systems on~$\A3\,$. The same method can
also be used to study the cohomology of the corresponding local systems
on~$\M3\,$. (As is well-known, this cohomology is very closely related to the
cohomology of the moduli spaces~$\M{3,n}$ of $n$-pointed curves of genus~$3$.)
If the local system is even, its cohomology on $\M3$ equals that on
$t(\M3)=\A3-\A{2,1}\,$,
the moduli space of indecomposable principally polarized abelian threefolds.
However, in the odd case, the cohomology can {\sl not} be explained in
terms of $\A3\,$. It is here that Teichm\"uller modular forms enter.
Our second main result (Theorem~\ref{topHodge}) gives the precise relationship between certain
spaces of Teichm\"uller cusp forms and the middle cohomology of
symplectic local systems on $\M{g}\,$. This is a partial analogue of the
results of Faltings and Chai~\cite{F-C} for~$\A{g}\,$.
We then specialize to genus~$3$, where we have determined these spaces
of Teichm\"uller cusp forms in a substantial number of cases, via
computations with concomitants. Finally, we discuss how these results
match perfectly with conjectural formulas for the `motivic' Euler
characteristics of the symplectic local systems of weight at most $20$
on~$\M3\,$, obtained from counts of curves over finite fields,
as in~\cite{BFvdG2}.
\end{section}

\section*{Acknowledgements}
We thank J. Bergstr\"om, D. Petersen, and F. Rodriguez Villegas
for useful discussions and remarks.
We thank G. Farkas, R. Pandharipande, and the Einstein Stiftung
in connection with the result in the appendix (\S\ref{appendix}).
All three authors thank the Max-Planck-Institut f\"ur Mathematik in Bonn
for the hospitality and excellent working conditions
and the referee for useful comments.
The third author also thanks YMSC at Tsinghua University for hospitality
enjoyed there.
Part of the research of the first author was supported by
the EPSRC grant EP/N031369/1.
The second author is supported by the Dutch Research Council (NWO),
grant 613.001.651.

%%%%%%%%%%%%%%%%%%%%%%%%%%%%%%%%%%%%%%%%
\begin{section}{Siegel Modular Forms}
Let ${\ZZ}^{2g}$ be the symplectic lattice of rank $2g$
with basis elements $e_1,\ldots, e_g,f_1,\ldots,f_g$
and with the symplectic pairing given by $\langle e_i,e_j\rangle=0=\langle f_i,f_j\rangle$ and $\langle e_i,f_j\rangle=\delta_{ij}$. 
We write
$\Gamma_g= {\rm Sp}(2g,{\ZZ})={\rm Aut}({\ZZ}^{2g}, \langle\, , \, \rangle)$
for the Siegel modular group of degree $g$. 
An element $\gamma \in \Gamma_g$
can be given as a $2\times 2$ matrix of $g\times g$ blocks with respect to the basis
$e_1,\ldots,e_g,f_1,\ldots,f_g$.
We will denote by $\A{g}$ the moduli stack of 
principally polarized abelian varieties
and by $\M{g}$ the moduli stack of curves of genus $g$ (for $g>1$). 
Over the complex numbers the orbifold $\A{g}({\CC})$, 
associated to the moduli space $\A{g}$, 
can be written as $\Gamma_g\backslash \mathfrak{H}_g$, where 
the Siegel modular group $\Gamma_g=\Sp{2g}{\ZZ}$ of degree $g$ 
acts on the Siegel upper half space 
$$
\mathfrak{H}_g=\{ \tau \in {\rm Mat}(g \times g, {\CC}):
\tau^t=\tau, {\rm Im}(\tau) > 0\}
$$ 
of degree $g$ in the usual way: 
$$
\tau \mapsto \gamma\cdot \tau =(a\tau+b)(c\tau+d)^{-1} \qquad \text{\rm for 
$\gamma=\left(\begin{matrix} a & b \\ c & d \\ \end{matrix} \right)
\in \Gamma_g\, .$}
$$
The moduli space $\A{g}$ carries a rank $g$ vector bundle,
the Hodge bundle ${\EE}$.
The induced bundle on  $\A{g}({\CC})$ corresponds 
to the factor of automorphy 
$$
j(\gamma,\tau)=c\tau +d \, .
$$
For an irreducible representation $\rho$ of $\GL{g}{\CC}$ of highest weight 
$(\rho_1, \ldots, \rho_g)$ with 
$\rho_1 \geq \rho_2 \geq \ldots \geq \rho_g$,
we have a corresponding vector bundle ${\EE}_{\rho}$ on $\A{g}$; 
the Hodge bundle corresponds to the standard representation 
with highest weight $(1,0,\ldots,0)$ 
and its determinant has highest weight $(1,\ldots,1)$. 
If $\rho: \GL{g}{\CC} \to {\rm GL}(W)$ is a finite-dimensional 
complex representation, the induced 
vector bundle on $\A{g}({\CC})$ is  defined by the
factor of automorphy
$$
j(\gamma, \tau)= \rho(c\tau+d)\, .
$$

A scalar-valued Siegel modular form of degree $g>1$ and 
weight $k$ is a holomorphic function
$f \colon \mathfrak{H}_g \to {\CC}$ satisfying 
$$
f(\gamma \cdot \tau)= \det(c\tau+d)^k f(\tau)
$$
for all $\gamma \in \Gamma_g$, while for $g=1$ 
we also need a growth condition at infinity.
If $W$ is a finite-dimensional complex vector space
and $\rho: \GL{g}{\CC} \to {\rm GL}(W)$
 a representation, then a vector-valued 
Siegel modular form of degree $g>1$ and weight $\rho$
is a holomorphic map $f: \mathfrak{H}_g \to W$ such that 
for all $\gamma \in \Gamma_g$
$$
f(\gamma\cdot \tau)= \rho(c\tau+d) f(\tau) \, .
$$
Siegel modular forms of weight $\rho$ can be interpreted 
as sections of the vector bundle ${\EE}_{\rho}\,$, and conversely.
Sections of the $k$th power $L^k$ of the determinant line bundle 
$L=\det({\EE})$ correspond to scalar-valued  Siegel modular forms
of degree $g$ and weight $k$. 
The vector bundle ${\EE}$ and
the bundles ${\EE}_{\rho}$ extend in a canonical way to 
Faltings-Chai type toroidal compactifications
of $\A{g}$ and the Koecher principle says that their sections do so too.

A Siegel modular form $F$ of weight $\rho$ admits a Fourier expansion
$$
F= \sum_{n} a(n) \, q^n \qquad 
\text{\rm with $q^n=e^{2\pi i \, {\rm Tr} (n \tau )}$},
$$
where $n$ runs over the half-integral symmetric positive semi-definite  
$g\times g$ matrices and $a(n) \in W$. (Half-integral means that $2n$
is integral with even entries on the diagonal.) 

We are interested in the case $g=3$. 
For an irreducible representation $\rho$
of highest weight $(\rho_1,\rho_2,\rho_3)$ of $\GL{3}{\CC}$ 
we denote the weight of the corresponding Siegel modular forms 
by 
$$
(i,j,k)= w(\rho)=
(\rho_1-\rho_2,\rho_2-\rho_3,\rho_3)\, . \eqno(1)
$$
The vector space of Siegel modular forms of weight $(i,j,k)$ on
$\Gamma_3$ is denoted by $M_{i,j,k}$. 
The space of cusp forms is denoted by $S_{i,j,k}$.
For scalar-valued Siegel modular forms we often abbreviate the weight $(0,0,k)$
by $k$.

Scalar-valued Siegel modular forms of degree $3$ form a graded ring:
$$
R=\oplus_{k=0}^{\infty} M_{0,0,k}\, .
$$
Vector-valued modular forms of degree $3$ form a graded module 
$M=\oplus_{i,j,k} M_{i,j,k}$
over the ring $R$.
The ring $R$ was described by
Tsuyumine \cite{Tsuyumine}. 
He gave $34$ generators and the generating function of $R$.
His work used results by Igusa \cite{Igusa1967} and Shioda \cite{Shioda}.
Igusa showed that there is 
an exact sequence
$$
0 \to \chi_{18} \, R \to R \, {\buildrel r \over \longrightarrow} \, I(2,8)\, ,
$$
where $\chi_{18}$ is a cusp form of weight $18$ (see Section \ref{chi18}) and 
$I(2,8)$ is the ring of invariants of binary octics. 
The map $r$ is induced by the restriction map to the zero locus
in $\A{3}$ of $\chi_{18}$. This locus is the closure of the 
image of the hyperelliptic locus ${\mathcal H}_3$ under the Torelli map.
Shioda determined the ring of invariants of binary octics \cite{Shioda}.
In the recent paper \cite{LerRit}, Lercier and Ritzenthaler reduce the
number of generators of $R$ to $19$. 

For Siegel modular forms, both the notion of degree and that of genus are used;
we tend to use degree for Siegel modular forms and genus for 
Teichm\"uller modular forms, but are not strict in this respect.

\end{section}
%%%%%%%%%%%%%%%%%%%%%%%%%%%%%%%%%%%%%%%
\begin{section}{The Scalar-valued Siegel Modular Form $\chi_{18}$}\label{chi18}
The scalar-valued Siegel modular form $\chi_{18}$ of degree $3$
is up to a normalization defined as the
product of the $36$ even theta constants. 
It is a cusp form of weight $18$. Its Fourier expansion starts with 
$$
\chi_{18}=-\left(\frac{\sigma_3-\sigma_2+\sigma_1-1}{\sigma_3}\right)^2
\, (\sigma_3^2-2\sigma_3\sigma_1+8\, \sigma_3+\sigma_1^2-4\, \sigma_2) \,
q_1^2q_2^2q_3^2 +\cdots
$$
where we use 
$$
q_j=e^{2 \pi i \tau_{jj}}, \quad u=e^{2 \pi i \tau_{12}}, \quad
v=e^{2 \pi i \tau_{13}}\quad \text{and}\quad  w=e^{2 \pi i \tau_{23}} \, , \eqno(2)
$$
and $\sigma_i$ is the $i$th elementary symmetric function in $u,v,w$.
Thus $\chi_{18}$ vanishes with multiplicity $2$ at infinity. 
It is well-known that the divisor of $\chi_{18}$ in the toroidal compactification $\tilde{\A{3}}$ is $H+2D$ with $H$ the 
hyperelliptic locus and $D$ the divisor of 
$\tilde{\mathcal A}_3$ at infinity:
$$
{\rm div}(\chi_{18})= H +2\, D \, .
\eqno(3)
$$
We now give a direct proof that $\chi_{18}$ is up to
a scalar the unique cusp form of weight $18$ vanishing twice at infinity.
This fact also follows from the result of Harris and Morrison 
\cite[Corollary~0.5]{Harris-Morrison} on the slope of effective divisors on 
$\overline{\mathcal{M}}_3$.

\begin{lemma}\label{uniquenesschi18} 
The space of Siegel cusp forms of degree $3$ and weight~$18$ vanishing with multiplicity $\geq 2$ on $D$ is generated by $\chi_{18}$.
\end{lemma}
\begin{proof} The dimension of $S_{0,0,18}$ is $4$. In order to construct a basis
we consider the Eisenstein series $E_4$ and $E_6$ whose Fourier expansions start with
$$
E_4= 1+240(q_1+q_2+q_3) +\cdots, \qquad
E_6=1-504(q_1+q_2+q_3) + \cdots
$$
and cusp forms $F_{12}$ and $F_{14}$ of weight $12$ and $14$; note that
$\dim S_{0,0,12}=1=\dim S_{0,0,14}$ (cf.~\cite{Tsuyumine}).
We normalize these forms such that the
following table gives their first Fourier coefficients:

\begin{footnotesize}
\smallskip
\vbox{
\bigskip\centerline{\def\quad{\hskip 0.6em\relax}
\def\quod{\hskip 0.5em\relax }
\vbox{\offinterlineskip
\hrule
\halign{&\vrule#&\strut\quod\hfil#\quad\cr
height2pt&\omit&&\omit &&\omit &&\omit  &\cr
%\noalign{\hrule}
& $n$ &&
$\det(2n)$ && $F_{12}$ && $F_{14}$  &\cr
\noalign{\hrule}
& $1_3$ &&
$8$ && $164$ && $20$ &\cr
& $A_1(1/2)\oplus A_2(1/2)$ &&
$6$ && $18$ && $-6$ &\cr
& $A_3(1/2)$ &&
$4$ && $1$ && $1$ &\cr
} \hrule}
}}
\end{footnotesize}

\noindent
Here we write the Fourier series of a modular form as before as
$$
\sum_{n\geq 0} a(n)\, q^n \qquad \text{\rm (using $q^n= e^{2 \pi i {\rm Tr}( n\tau)}$)}
$$
with $n$ running over the half-integral positive semi-definite symmetric 
matrices; $A_i$ refers to the Gram matrix of the standard root lattice.
Recall that
for a scalar-valued modular form we have $a(n)=a(u^t n u)$ for all
$u \in \GL{3}{\ZZ}$. We look at the induced action of $\GL{3}{\ZZ}$.
The set ${\mathcal N}$ of half-integral positive definite symmetric matrices
with $1$'s on the diagonal contains elements from three distinct orbits:
the orbit of $1_3$ (one element), the orbit of
$A_1(1/2)\oplus A_2(1/2)$ (sixteen elements), and
the orbit of $A_3(1/2)$ (six elements).
This implies that we can read off the Fourier coefficient of $q_1q_2q_3$ of 
$F_{12}$ and $F_{14}$ from the table above.
We claim that a basis of $S_{0,0,18}$ is given by 
$$
\chi_{18}, \, \chi_{18}{|T_2}, \,
E_4F_{14},  \,  E_6F_{12}
$$
with $T_2$ the Hecke operator of the prime $2$.
We can calculate the Fourier coefficient of $\chi_{18}{|T_2}$
using \cite{C-vdG}. 
We get the following Fourier coefficients:

\begin{footnotesize}
\smallskip
\vbox{
\bigskip\centerline{\def\quad{\hskip 0.6em\relax}
\def\quod{\hskip 0.5em\relax }
\vbox{\offinterlineskip
\hrule
\halign{&\vrule#&\strut\quod\hfil#\quad\cr
height2pt&\omit&&\omit &&\omit &&\omit &&\omit &&\omit  &\cr
%\noalign{\hrule}
& $n$ &&
$\det(2n)$ && $E_6\, F_{12}$ && $E_4\, F_{14}$ && $\chi_{18}\vert T_2$ && $\chi_{18}$ &\cr
\noalign{\hrule}
& $1_3$ &&
$8$ && $164$ && $20$ && $108$ && $0$ &\cr
& $A_1(1/2)\oplus A_2(1/2)$ &&
$6$ && $18$ && $-6$ && $0$ && $0$ &\cr
& $A_3(1/2)$ &&
$4$ && $1$ && $1$ && $-1$ && $0$ &\cr
} \hrule}
}}
\end{footnotesize}

\noindent
This shows that the four cusp forms of weight $18$  
are linearly independent and  that $\chi_{18}$
is up to a nonzero scalar the unique cusp form of weight $18$ 
that vanishes twice at infinity. 
\end{proof}

\begin{remark}
In \S11, we will obtain a considerably more general result,
with a different method of proof (Theorem~\ref{the_isom} and
Corollary~\ref{chi18b}).
\end{remark}
\end{section}
%%%%%%%%%%%%%%%%%%%%%%%%%%%%%%%%%%%%%%%%%%%%%%%%%%%%%%%%%%%%%%%%%%%%%%
\begin{section}{The Vector-valued Siegel Modular Form $\chi_{4,0,8}$}\label{chi408}
The modular form generating $S_{4,0,8}$ is the unique nonzero cusp form (up to scalar)
for which $i+2j+3k$ is minimal (equal to $28$) (cf.~\cite{Taibi}).
This form $\chi_{4,0,8}$  will play a central role in this paper.
Its Fourier expansion can be obtained as in~\cite{C-vdG}
by taking the Schottky form of weight~$8$
on $\Gamma_4$ and developing it in the normal directions to $\mathfrak{H}_1 \times
\mathfrak{H}_3$ in $\mathfrak{H}_4$. The lowest order term gives a 
nonzero multiple of $\Delta \otimes \chi_{4,0,8}$ in $S_{12}(\Gamma_1) \otimes
S_{4,0,8}(\Gamma_3)$.  We normalize $\chi_{4,0,8}$  so that its Fourier 
expansion starts as follows
$$
\left( \begin{smallmatrix}
0\\ 0 \\ 0 \\
(v-1)^2(w-1)^2/vw \\
(u-1)(v-1)(w-1)(-1+1/vw+1/uw-1/uv) \\
(u-1)^2(w-1)^2/uw \\
0\\
(u-1)(v-1)(w-1)(-1+1/vw-1/uw+1/uv) \\
(u-1)(v-1)(w-1)(-1-1/vw+1/uw+1/uv)\\
0\\ 0 \\ 0 \\
(u-1)^2(v-1)^2/uv \\
0\\ 0\\
\end{smallmatrix}
\right )q_1q_2q_3+\cdots \, ,
$$
where we use the same variables as in (2).
This modular form is a lift of $\Delta=\sum \tau(n)q^n
 \in S_{12}(\Gamma_1)$ 
and its Hecke eigenvalue at a prime $p$ is by the results of~\cite{BFvdG2} predicted to be
$$
\tau(p)\, (p^5+\tau(p)+p^6).
$$
%for $T_p$ with $p$ prime.

We embed $\mathfrak{H}_2\times \mathfrak{H}_1$ in $\mathfrak{H}_3$
via 
$$
(\tau^{\prime},\tau^{\prime\prime}) \mapsto 
\left( \begin{matrix} \tau^{\prime} & 0 \\ 0 & \tau^{\prime\prime}\\ 
\end{matrix} \right)
$$
and consider the vanishing of $\chi_{4,0,8}$ along this locus.
\begin{lemma}\label{I2notI3} 
Let ${\mathcal I} \subset {\mathcal O}_{\mathfrak{H}_3}$ 
be the
ideal sheaf of $\mathfrak{H}_2\times \mathfrak{H}_1$ in $\mathfrak{H}_3$.
The 
coordinates of the form $\chi_{4,0,8}$ 
lie in ${\mathcal I}^2$, but not all in~${\mathcal I}^3$.
\end{lemma}
\begin{proof}
Consider the Schottky form $J_8$,
a scalar-valued Siegel cusp form of weight $8$
and degree $4$. If we restrict it to $\mathfrak{H}_2\times
\mathfrak{H}_2$ we find as first term in its development in the
normal directions to $\mathfrak{H}_2\times \mathfrak{H}_2$ the tensor product
$\chi_{6,8} \otimes \chi_{6,8}$, where $\chi_{6,8}$ is a Siegel cusp form
of weight $(6,8)$ in degree $2$. If we develop $J_8$ along
$\mathfrak{H}_3\times \mathfrak{H}_1$  then we find as first
term in the normal directions $\chi_{4,0,8}\otimes
\Delta$ with $\Delta$ the elliptic modular cusp form of weight $12$.
We refer to \cite{C-vdG} for these facts. Comparing the degrees of vanishing
along $\mathfrak{H}_3\times \mathfrak{H}_1$ and $\mathfrak{H}_2\times
\mathfrak{H}_2$, see~\cite[Section 5]{C-vdG},  we see that $\chi_{4,0,8}$ vanishes
with multiplicity~$2$ along $\mathfrak{H}_2 \times \mathfrak{H}_1$.

Another way to see this is by looking at the expansion of $\chi_{4,0,8}$
given in \cite[p.\ 39]{C-vdG}.  We thus see that
the coordinates of
$\chi_{4,0,8}$ 
lie in ${\mathcal I}^2$, but not all in ${\mathcal I}^3$.
\end{proof}

We can trivialize the pullback of ${\EE}$ to $\mathfrak{H}_3$ as $\mathfrak{H}_3 \times {\CC}^3$.
We choose coordinates $z_1,z_2,z_3$ on ${\CC}^3$.
Since the pullback of 
$\Sym^2({\EE})$ can be identified with the cotangent bundle of $\mathfrak{H}_3$
the coordinates $\tau_{ij}$ correspond to $z_i  z_j$.
In particular, a basis of ${\Sym}^4({\EE})$ corresponds to the monomials of degree $4$
in $z_1,z_2$ and $z_3$. Therefore, we can write $\chi_{4,0,8}$ as
$$
\chi_{4,0,8}=\sum_I n_I \alpha_I \, z^I \, ,
$$
where for a multi-index $I=(i_1,i_2,i_3)$ we inserted a normalization factor 
$n_I={4!}/{i_1!i_2!i_3!}$.
We take the sum over the lexicographically ordered monomials 
$z^I=z_1^{i_1}z_2^{i_2}z_3^{i_3}$ of degree $4$ and $\alpha_I$ is a holomorphic
function on $\mathfrak{H}_3$ which we can present as
a Fourier series living in ${\CC}[u^{\pm 1},v^{\pm 1}, w^{\pm 1}][[q_1,q_2,q_3]]$.
The $z^I$ are just dummy variables to indicate the coordinates 
 $n_I \, \alpha_I$ of the vector-valued modular form
$\chi_{4,0,8}$.

\medskip
The symmetric group $\mathfrak{S}_3$ acts on ${\CC}^3$ by permuting $z_1,z_2, z_3$ 
and hence induces an action on $\mathfrak{H}_3$ via
$$
(\tau_{ij}) \mapsto (\tau_{\sigma(i) \sigma(j)}) \qquad \sigma \in \mathfrak{S}_3\, .
\eqno(4)
$$
This action of $\mathfrak{S}_3$ derives from an action of a subgroup
of $\Gamma_3$ by taking a $3\times 3$ permutation matrix $a=(a^{-1})^t$ and using 
$(\begin{smallmatrix} a & 0 \\ 0 & a \\ \end{smallmatrix})$.
To give the action on the Fourier expansion,
note that we have an induced action of $\mathfrak{S}_3$ on $q_1,q_2,q_3$ 
and on $u,v,w$.
\begin{lemma}\label{actionS3}
The action of $\mathfrak{S}_3$ on $\{1,2,3\}$ 
induces an action on the coordinates 
$\alpha_I$ of $\chi_{4,0,8}$ given by
$$
\alpha_I(q_1,q_2,q_3,u,v,w) 
\mapsto 
\alpha_{\sigma(I)}
\left( q_{\sigma(1)},q_{\sigma(2)}, q_{\sigma(3)},\sigma(u),\sigma(v),\sigma(w)
\right)\, .
$$
\end{lemma}
The action on the fifteen coordinates has 
one orbit of length $6$ and three orbits of length $3$.
The reader may check that the action of 
$(23)\in \mathfrak{S}_3$ on ${\Sym}^4({\EE})$ 
is given by sending the transpose of $v_I=(v_1,v_2,\ldots,v_{15})$
to the transpose of
$$
(v_1,v_3,v_2,v_6,v_5,v_4,v_{10},v_9,v_8,v_7,v_{15},v_{14},v_{13},v_{12},v_{11})
$$
and that of $(13)\in \mathfrak{S}_3$ is given by
$$
(v_{15},v_{14},v_{10},v_{13}, v_{9},v_{6},v_{12},v_8,v_5,v_3,v_{11}, v_7,v_4,v_2,v_1).
$$ 
For $\chi_{4,0,8}$ it
 thus suffices to give the coordinates $v_1,v_2,v_4$ and $v_5$ only;
the other coordinates can then be calculated by using the action of
Lemma~\ref{actionS3}.
\end{section}
%%%%%%%%%%%%%%%%%%%%%%%%%%%%%%%%%%%%%%%
\begin{section}{The Fourier-Jacobi Expansion of the Schottky Form}
\label{FJ}
In the preceding section the form $\chi_{4,0,8}$ was defined by
developing the Schottky form of degree $4$ and weight $8$
along $\mathfrak{H}_1 \times \mathfrak{H}_3$ and observing that the
first nonzero term is $\Delta \otimes \chi_{4,0,8}$ in $S_{12}(\Gamma_1)
\tensor S_{4,0,8}(\Gamma_3)$. For our application we need to be able
to calculate the Fourier expansion of $\chi_{4,0,8}$ quite far
and for that we need the first Fourier-Jacobi coefficient of the scalar-valued 
form $J_8$. 

Recall that the Schottky form $J_8$ can be expressed as follows:
\[
J_8=(R_{00}^2+R_{01}^2+R_{10}^2-2(R_{00}R_{01}+R_{00}R_{10}+R_{01}R_{10}))/2^{16}\, ;
\]
here we write
\[
R_{\mu \nu}(\tau)=\prod_{\alpha,\, \beta,\, \gamma\in \left\{0,1\right\}} 
\vartheta_
{
\left[
\begin{smallmatrix}
\mu & 0 & 0 & 0 \\
\nu  & \alpha & \beta & \gamma
\end{smallmatrix}
\right]
}
(\tau),
\quad
(\tau \in  \mathcal{H}_4)
\]
where we use the theta series with characteristics
$
\left[
\begin{smallmatrix}
a \\
b
\end{smallmatrix}
\right]
$
with $a$ and $b$  row vectors of size $g\in \ZZ_{\geq 1}$ 
with coordinates in $\ZZ$:
\[
\vartheta_{
\left[
\begin{smallmatrix}
a \\
b
\end{smallmatrix}
\right]
}
(\tau,z)=
\sum_{l\in \ZZ^g}
e^{\pi i (l+a/2) \tau (l+a/2)^t}
e^{2 \pi i (l+a/2)\cdot (z+b/2)^t},
\quad 
\tau \in \mathcal{H}_g,
\quad 
z=(z_1,\ldots,z_g)\in \CC^g.
\]
When $z=0$, we denote 
$
\vartheta_{
\left[
\begin{smallmatrix}
a \\
b
\end{smallmatrix}
\right]
}
(\tau,0)
$
simply by
$
\vartheta_{
\left[
\begin{smallmatrix}
a \\
b
\end{smallmatrix}
\right]
}
(\tau).
$
We thus need to calculate the first Fourier-Jacobi coefficient of
these theta functions. For any $\tau=(\tau_{ij}) \in \mathfrak{H}_4$
we write
$$
\tau= \left(\begin{matrix} \tau' & z^t \\ z & \tau_4\\ \end{matrix}\right)
\quad 
\text{where\, }
\tau' \in \mathfrak{H}_3 \quad 
\text{and\, }
z=(\tau_{14},\tau_{24},\tau_{34})
$$
and we abbreviate $\tau_{ii}$ as $\tau_i$.
We write the Fourier-Jacobi expansion of
$
\vartheta_
{
\left[
\begin{smallmatrix}
\mu & 0 & 0 & 0 \\
\nu  & \alpha & \beta & \gamma
\end{smallmatrix}
\right]
}
$
as
\[
\vartheta_{
\left[
\begin{smallmatrix}
\mu & 0 & 0 & 0 \\
\nu  & \alpha & \beta & \gamma
\end{smallmatrix}
\right]
}
(\tau)=
\sum_{l\in \ZZ}
(-1)^{l\, \gamma}
\vartheta_{
\left[
\begin{smallmatrix}
\mu & 0 & 0 \\
\nu  & \alpha & \beta 
\end{smallmatrix}
\right]
}
(\tau',l\, z) 
e^{\pi i \, l^2 \tau_4}.
\eqno(5)
\]
So we get
$$
\vartheta_{
\left[
\begin{smallmatrix}
\mu & 0 & 0 & 0 \\
\nu  & \alpha & \beta & \gamma
\end{smallmatrix}
\right]
}
(\tau)=
\vartheta_{
\left[
\begin{smallmatrix}
\mu & 0 & 0 \\
\nu  & \alpha & \beta
\end{smallmatrix}
\right]
}
(\tau')+
(-1)^{\gamma}
(
\vartheta_{
\left[
\begin{smallmatrix}
\mu & 0 & 0 \\
\nu  & \alpha & \beta
\end{smallmatrix}
\right]
}
(\tau',z)+
\vartheta_{
\left[
\begin{smallmatrix}
\mu & 0 & 0 \\
\nu  & \alpha & \beta
\end{smallmatrix}
\right]
}
(\tau',-z) 
)\, 
q_4^{1/2}+O(q_4^2) \, ,
$$
where $q_4=e^{2\pi i \tau_4}$. 
These first terms correspond to $l=0$ and $l=\pm 1$ in
formula (5) and the next term is given by $l=\pm 2$ 
which gives rise to $O(q_4^2)$.
Since we are dealing with even characteristics, the latter 
formula simplifies to
\[
\vartheta_{
\left[
\begin{smallmatrix}
\mu & 0 & 0 & 0 \\
\nu  & \alpha & \beta & \gamma
\end{smallmatrix}
\right]
}
(\tau)=
\vartheta_{
\left[
\begin{smallmatrix}
\mu & 0 & 0 \\
\nu  & \alpha & \beta
\end{smallmatrix}
\right]
}
(\tau')+
2\,(-1)^{\gamma}
\vartheta_{
\left[
\begin{smallmatrix}
\mu & 0 & 0 \\
\nu  & \alpha & \beta
\end{smallmatrix}
\right]
}
(\tau',z)\, 
q_4^{1/2}+O(q_4^2)\, . \eqno(5')
\]
We introduce for $\tau \in \mathfrak{H}_3$ and  $z \in {\CC}^3$
the notation
\[
r_{\mu \nu}(\tau)=\prod_{\alpha,\, \beta \in \left\{0,1\right\}} 
\vartheta_
{
\left[
\begin{smallmatrix}
\mu & 0 & 0 \\
\nu  & \alpha & \beta
\end{smallmatrix}
\right]
}
(\tau),
\qquad
s_{\mu \nu}(\tau,z)=\sum_{\alpha,\, \beta \in \left\{0,1\right\}} 
\left(\frac
{
\vartheta_
{
\left[
\begin{smallmatrix}
\mu & 0 & 0 \\
\nu  & \alpha & \beta
\end{smallmatrix}
\right]
}
(\tau,z)
}
{
\vartheta_
{
\left[
\begin{smallmatrix}
\mu & 0 & 0 \\
\nu  & \alpha & \beta
\end{smallmatrix}
\right]
}
(\tau)
}\right)^2
\quad
 \, .
\]
By using formulas $(5)$ and $(5')$ and this notation we get the following 
result.
\begin{lemma}
The Fourier-Jacobi expansion of $R_{\mu \nu}$ starts with
\[
R_{\mu \nu}(\tau)=
r_{\mu \nu}(\tau')^2
(
1-4\, s_{\mu \nu}(\tau',z)\, q_4
)+O(q_4^2).
\]
\end{lemma}

Since we are dealing with explicit theta series 
this allows us to compute the first Fourier-Jacobi coefficient of $J_8\,$.
We write the Fourier-Jacobi development as 
$J_8(\tau)=\sum_{m=0}^{\infty}
\varphi_{8,m}(\tau^{\prime},z) e^{2\pi i m \tau_4}$.
\begin{proposition}
The first nonzero Fourier-Jacobi coefficient of the Schottky form 
$J_8$ 
is given by
\[
\varphi_{8,1}(\tau',z)=\frac{1}{2^{12}}\, (r_{00}r_{01}r_{10})\,
(
-r_{00}s_{00}+
r_{01}s_{01}+
r_{10}s_{10}
)(\tau',z).
\]
\end{proposition}

%\begin{proof}	
\noindent
\emph{Proof.}
We have
\begin{align*}
J_8(\tau)=
\frac{1}{2^{16}}(\varphi_{8,0}(\tau')+\varphi_{8,1}(\tau',z)\, q_4+O(q_4^2)).
\end{align*}
The Fourier-Jacobi coefficient $\varphi_{8,0}$ is a modular form of weight 8 on $\Gamma_3$
but, since $J_8$ is a cusp form, we have 
$\varphi_{8,0}\in S_8(\Gamma_3)$ and this space is $(0)$, so
$\varphi_{8,0}=0$. Computing its expression in terms of the 
$r_{\mu \nu}$ gives
\begin{align*}
\varphi_{8,0}(\tau')&=
(r_{00}^4+r_{01}^4+r_{10}^4-2(r_{00}^2r_{01}^2+r_{00}^2r_{10}^2+r_{01}^2r_{10}^2))(\tau')\\
&=(r_{00}-r_{01}-r_{10})(r_{00}+r_{01}+r_{10})(r_{00}+r_{10}-r_{01})(r_{00}+r_{01}-r_{10})(\tau')\\
&=0
\end{align*}
since we have the following relation
\[
r_{00}-r_{01}-r_{10}=0. \eqno(6)
\]
The Fourier-Jacobi coefficient $\varphi_{8,1}$ is a Jacobi cusp form of weight 8 and index 1
on $\Gamma_3$ and we have
\begin{align*}
\varphi_{8,1}&=8
(-
r_{00}^2(r_{00}^2-r_{01}^2-r_{10}^2)s_{00}+
r_{01}^2(r_{00}^2-r_{01}^2+r_{10}^2)s_{01}+
r_{10}^2(r_{00}^2+r_{01}^2-r_{10}^2)s_{10})\\
&=16\,
r_{00}r_{01}r_{10}(-r_{00}s_{00}+
r_{01}s_{01}+
r_{10}s_{10})
\end{align*}
since the relation (6) among the $r_{\mu\nu}$ implies
%\begin{align*}
$$
r_{00}^2-r_{01}^2-r_{10}^2=2\, r_{01}r_{10}\,, \qquad
r_{00}^2-r_{01}^2+r_{10}^2=2\, r_{00}r_{10}\,, \qquad
r_{00}^2+r_{01}^2-r_{10}^2=2\, r_{00}r_{01}\,. \hfill \qed	
$$
%\end{align*}
%\end{proof}

Recall now that we write $\chi_{4,0,8}$ as $\sum n_I \alpha_I \, z^I$
where the sum is over the fifteen lexicographically ordered 
multi-indices of the monomials of
degree $4$ in the three variables $z_1,z_2,z_3$ and each $\alpha_I$
is a holomorphic function on $\mathfrak{H}_3$. 

\begin{proposition}
The coefficient $n_I \alpha_I$ for a multi-index $I=(i_1,i_2,i_3)$ 
with $i_1+i_2+i_3=4$
of the cusp form $\chi_{4,0,8}$ of weight 
$(4,0,8)$ on $\Gamma_3$ 
is given by
\[
\alpha_I= \frac{1}{2^{19}\, 3} r_{00}r_{01}r_{10}(\tau) 
\left(
-r_{00}(\tau)\frac{\partial^4}{\partial_I} s_{00}(\tau,0)
+r_{01}(\tau)\frac{\partial^4}{\partial_I} s_{01}(\tau,0)
+r_{10}(\tau)\frac{\partial^4}{\partial_I} s_{10}(\tau,0)
\right)\, ,
\]
where
${\partial^4}/{\partial_I}$ stands for $\partial^4/
%\partial z_1^{i_1} \partial z_2^{i_2} \partial z_3^{i_3}\,$
%with $z_1=\tau_{14}, z_2=\tau_{24}, z_3=\tau_{34}\,$.
\partial {\tau}_{14}^{i_1} \partial {\tau}_{24}^{i_2} 
\partial {\tau}_{34}^{i_3}\,$
and $n_I= 24/{i_1! \, i_2! \, i_3!}\,$.
\end{proposition}

Using Lemma~\ref{actionS3} and the remarks thereafter we see
that we need to calculate the coefficients for the multi-indices
$I=(1,1,1,1)$, $(1,1,1,2)$,
$(1,1,2,2)$  and $(1,1,2,3)$ 
only.

As a check on our calculations we use the fact that the Hecke 
eigenvalues of $\chi_{4,0,8}$ for a prime $p$ are 
by the results of~\cite{BFvdG2} predicted to be 
$$
\lambda_p= \tau(p) \, (p^5+\tau(p)+p^6)
$$
with $\Delta=\sum \tau(n)q^n \in S_{12}(\Gamma_1)$. By applying the 
explicit formulas for the Hecke operators for $\Gamma_3$ 
given in \cite{C-vdG}
we can calculate the Fourier coefficients $a(n)$ and hence the
eigenvalues for small $p$; for example with the notation used
in~\cite[p.~45]{C-vdG}
we have for $n_0=[1\, 1\, 1; 0\, 0\, 0]$
and $p=2$
%\begin{footnotesize}
$$
\begin{aligned}
a_2(n_0)= a(2\, n_0)+ 2^{15}\Sy^{-4}
\left(
\begin{smallmatrix}
1 & 0 & 1  \\
0 & 1 & 1  \\
0 & 0 & 2 
\end{smallmatrix}
\right)
a([1\, 1\, 2\,; 1\, 2\, 2]) + 2^{9}(\Sy^{-4}
\left(
\begin{smallmatrix}
1 & 0 & 1  \\
0 & 2 & 0  \\
0 & 0 & 2 
\end{smallmatrix}
\right)
a([1\, 2\, 2\,; 2\, 0\, 0]) & \\
{}+\Sy^{-4}
\left(
\begin{smallmatrix}
0 & 1 & 1  \\
2 & 0 & 0  \\
0 & 0 & 2 
\end{smallmatrix}
\right)
a([1\, 2\, 2\,; 0\, 2\, 0])) &.
\end{aligned}
$$
%\end{footnotesize}
The Fourier coefficients needed are
the initial one
$$
\begin{aligned}
a(n_0)=a([1\, 1\, 1\,; 0\, 0\, 0])&=[0, 0, 0, 4, 0, 4, 0, 0, 0, 0, 0, 0, 4, 0, 0]^t\\
\end{aligned}
$$
and  
%\begin{footnotesize}
$$\begin{aligned}
a([2\, 2\, 2\,; 0\, 0\, 0])&=[-512, 0, 0, -2816, 0, -2816, 0, 0, 0, 0, -512, 0, -2816, 0, -512]^t \\
a([1\, 1\, 2\,; 1\, 2\, 2])&=[0, 0, 0, 0, 1, 1, 0, 1, 3, 2, 0, 0, 1, 2, 1]^t \\
a([1\, 2\, 2\,; 0\, 2\, 0])&=[0, 0, 0, -24, 0, -48, 0, -48, 0, -96, 48, 0, -48, 0, -48]^t \\
a([1\, 2\, 2\,; 2\, 0\, 0])&=[0, 0, 0, -48, 0, -24, -96, 0, -48, 0, -48, 0, -48, 0, 48]^t;
\end{aligned}
$$
%\end{footnotesize}
this gives $\lambda_2(\chi_{4,0,8})=-1728=-24\, (2^5-24+2^6)$ as expected.
We similarly checked the agreement for
$\lambda_3$ and $\lambda_5\,$.
\end{section}

%%%%%%%%%%%%%%%%%%%%%%%%%%%%%%%%%%%%%%%%
\begin{section}{Another expression for $\chi_{4,0,8}$}
In this section we derive another expression for $\chi_{4,0,8}$ that
connects it to the geometry of curves of genus three.
For a principally polarized abelian variety $(X,\Theta)$ of dimension~$g$
we consider the space
$$
\Gamma_{00}(X,\Theta):=\{ s \in H^0(X, {\mathcal O}(2\Theta)):
m_0(s) \geq 4\} \, ,
$$
where $m_0$ means the order of vanishing at the origin, see \cite{vGvdG}.
For a principally polarized abelian variety $(X, \Theta)$ that is not
decomposable (that is, not a product of nontrivial
lower-dimensional principally polarized abelian subvarieties)
it is known that
$$
\dim \Gamma_{00}(X, \Theta)=2^g-\frac{g(g+1)}{2}-1 \, .
$$
If for  $\sigma \in ({\ZZ}/2{\ZZ})^g$
we define the second order theta function
$\Theta_{\sigma}(\tau,z)$ on $\mathfrak{H}_g \times {\CC}^g$ by
$$
\Theta_{\sigma}(\tau,z)=
\vartheta_{[\begin{smallmatrix} \sigma \\ 0 \\ \end{smallmatrix}]}(2\, \tau,2\, z)
$$
%\sum_{m \in {\ZZ}^g} \exp{[2\pi i \left((m+\sigma/2)^t\tau (m+\sigma/2) + 2
%(m+\sigma/2)^t z\right)]}\, ,
%$$
then for fixed $\tau$ the
$\Theta_{\sigma}(\tau,z)$ with $\sigma \in ({\ZZ}/2{\ZZ})^g$
give a basis of
$H^0(X_{\tau},{\mathcal O}(2\Theta))$ with $X_{\tau}={\CC}^g/\Lambda_{\tau}$
the abelian variety
associated to $\tau$.
Now $s=\sum_{\sigma} a_{\sigma} \Theta_{\sigma}(\tau,z)$
belongs to $\Gamma_{00}$ if and only if
$$
\sum_{\sigma} a_{\sigma} \Theta_{\sigma}(\tau,0)=0,
\quad
\sum_{\sigma} a_{\sigma} \frac{\partial^2 \Theta_{\sigma}}
{\partial z_i \partial z_j}(\tau,0)=0 \qquad
\text{\rm for all $1\leq i, j \leq g$}.
$$
For an indecomposable principally polarized abelian threefold
the space $\Gamma_{00}$ has dimension~$1$.
In \cite{Frobenius} Frobenius constructed a non-zero element of  $\Gamma_{00}$
by defining a function $\varphi(\tau,z)$ on $\mathfrak{H}_3\times {\CC}^3$.
A somewhat different construction, leading to the same result,
is as follows. Let
$\Gamma_3[2,4]$ be the usual congruence subgroup of $\Gamma_3$ of level $(2,4)$.

The theta functions of the second order $\Theta_{\sigma}(\tau,0)$ on
$\mathfrak{H}_3$
define a morphism
$$
{\rm Th} \colon {A}_3[2,4]=\Gamma_3[2,4]\backslash \mathfrak{H}_3 \to {\PP}^7\, ,
$$
which is an injective map and an immersion
along the locus $A^0_3[2,4]$ of indecomposable abelian varieties,
see \cite[Theorem~3.2]{Sasaki}.
The closure of the image is a hypersurface given by an
explicit equation $F(\ldots,x_{\sigma},\ldots)=0$ of degree $16$
in the coordinates $x_{\sigma}$ of ${\PP}^7$, see \cite[p.\ 632]{vGvdG}.
Then the expression
$$
\phi(\tau,z) := \sum_{\sigma} \frac{\partial F}{\partial x_{\sigma}}
(\ldots, \Theta_{\sigma}(\tau,0),
\ldots) \, \Theta_{\sigma}(\tau,z)
$$
equals up to a non-zero multiplicative constant
the function constructed by Frobenius, see \cite[p.\ 624]{vGvdG}. In fact,
the functions
$\varphi(\tau,z)$ and $\phi(\tau,z)$ differ multiplicatively by a function
that depends only on $\tau$ and  is invariant under $\Gamma_3[2,4]$
and descends to a holomorphic function on $A^0_3[2,4]$
and thus is constant since the locus of decomposable abelian threefolds has codimension $2$.
For a $\tau$ that corresponds to the periods of a smooth projective nonhyperelliptic
curve $Y$ the first term of $\phi$ in the Taylor expansion
as a function in $z_1,z_2,z_3$
gives the quartic polynomial that defines the canonical image of $Y$ as Frobenius
showed \cite[p.\ 37]{Frobenius}.
In fact, the zero locus of $\phi$ on the abelian variety
${\rm Jac}(Y)={\CC}^3/\Lambda_{\tau}$ is the surface $\{(x-y): x, y \in Y\}$.
Note that for $g=3$ the divisor $\{ (x-y)\in {\rm Jac}(Y): x,y \in Y\}$
belongs to $|2\Theta|$.
This surface is singular since under the map
$Y\times Y \to {\rm Jac}(Y)$ given by $(x,y) \mapsto (x-y)$
the diagonal is contracted and the tangent cone
to the singularity is the cone over the canonical image of $Y$, see \cite{vGvdG}.
For yet another description of $\phi(\tau,z)$ we refer to \cite[Proposition~1]{G-SM}.

It follows from the properties of the theta functions $\Theta_{\sigma}$
and the fact that $\phi$ vanishes to order $4$ along the zero section
in the universal abelian threefold $\X{3}$ over $\A{3}$
that the degree $4$ term in the normal expansion of $\phi$ along the
zero section 
defines a section of
$\Sym^4({\EE})\otimes \det^8({\EE})$.

\begin{proposition}\label{Frobeniusphi0}
The function $\phi$ defines a holomorphic Siegel modular cusp form of weight $(4,0,8)$
and for a nonhyperelliptic Jacobian ${\rm Jac}(Y)$
the naturally associated vector of length $15$ gives the
coefficients of the quartic defining the canonical image of $Y$.
\end{proposition}
\begin{proof}
Over the locus of indecomposable abelian threefolds the section $\phi$ vanishes
along the zero section of ${\mathcal X}_3\to \A{3}$ to order $4$.
Since its zero locus for a  smooth curve $Y$ is the surface $\{(x-y): x,y\in Y\}$
it vanishes with exact order $4$.
Therefore
the first term in the development is a non-zero section of
$\Sym^4({\EE}) \otimes \det^8({\EE})$ over this locus.
Since the complement has codimension two, the section extends
over all of $\A{3}$ 
to give a holomorphic modular form. It vanishes on the locus
of decomposable abelian varieties, hence it defines a cusp form.
\end{proof}
\end{section}
%%%%%%%%%%%%%%%%%%%%%%%%%%%%%%%%%%%%%%%%

%%%%%%%%%%%%%%%%%%%%%%%%%%%%%%%%%%%%%%%
% SEVEN
\begin{section}{Teichm\"uller Modular Forms}\label{TMF}
The Torelli map is a morphism $t=t_g\colon \M{g} \to \A{g}$
of Deligne-Mumford stacks obtained
by associating to a smooth projective curve of genus $g$ 
its Jacobian. It can be extended to a morphism from the Deligne-Mumford
compactification $\barM{g}$
to the toroidal compactification defined by the second Voronoi
fan~\cite{Namikawa}.
In the case of genus $3$ this compactification $\tilde{\mathcal A}_3$
is the standard compactification of~$\A{3}$
(cf.~\cite{DSHS}).

The moduli space $\barM{3}$ carries a vector bundle ${\EE}^{\prime}$
of rank $3$, the Hodge bundle. 
It is the pullback under the Torelli morphism of the Hodge bundle
${\EE}$ on $\tilde{\mathcal A}_3$.
Moreover, for each irreducible representation $\rho$ of ${\rm GL}(3)$
of highest weight $(\rho_1,\rho_2,\rho_3)$ we have the vector bundle 
${\EE}_{\rho}^{\prime}$ on $\barM{3}$, 
again pulled back from $\tilde{\mathcal A}_3$. 

A Teichm\"uller modular form of weight 
$w(\rho)=(\rho_1-\rho_2,\rho_2-\rho_3,\rho_3)$ and degree $3$
is a section of ${\EE}_{\rho}^{\prime}$ on
$\M{3}$. In Proposition~\ref{extensiontoM3bar} below, we show that
a Teichm\"uller modular form extends to $\barM{3}$.
We denote the space of Teichm\"uller modular forms of weight $w(\rho)$ 
by $T_{w(\rho)}$. If $w(\rho)=(0,0,k)$ we are dealing with scalar-valued
Teichm\"uller modular forms (of weight $k$). Scalar-valued 
Teichm\"uller modular forms form a ring $T=\oplus_k T_{0,0,k}$.

Ichikawa showed that there exists a Teichm\"uller
modular form $\chi_{9}$ of weight $(0,0,9)$ on $\barM{3}$.
It is a square root of the pullback of $\chi_{18}$.
In particular, $\chi_9$ vanishes on the hyperelliptic
locus, with class $h\sim 9\lambda-\delta_0-3\delta_1$. The divisor
of $\chi_9$ is therefore known:
$$
{\rm div}(\chi_9)=h+\delta_0+3\, \delta_1\, . \eqno(7)
$$
Ichikawa also showed that
the ring $T$ is generated over the ring $R$ of scalar-valued Siegel
modular forms of degree $3$ by $\chi_9$. 
See \cite{Ichikawa1,Ichikawa2,Ichikawa3,Ichikawa4}.

The map $\mathcal{M}_3 \to \mathcal{A}_3$ of stacks is a double
cover of its image, the locus of indecomposable abelian
threefolds. It is ramified along the hyperelliptic locus. 
The covering involution $\iota$ can be understood by using
a fine moduli space $\M3[N]$ of curves with level-$N$
structure ($N\ge3$): 
it sends the isomorphism class of a pair $(C,\alpha)$
of a curve $C$ and a level structure $\alpha$ to the class
of $(C,-\alpha)$. Note that $(A,\alpha)\cong (A,-\alpha)$
for a polarized abelian variety $A$.

Further, $\iota$ preserves the Hodge bundle $\EE'$ and acts
as $-1$ on the fiber. Hence the bundle $\EE_{\rho}'$
is preserved as well and the action on the fiber
is as $(-1)^{\rho_1-\rho_2+\rho_3}=(-1)^{\rho_1+\rho_2+\rho_3}$.

Call $\rho$ odd (resp.~even) when $\rho_1+\rho_2+\rho_3$
is odd (resp.~even). 
Let us also call a Teichm\"uller (or Siegel) modular form of weight $w(\rho)$
odd (resp.~even) when $\rho$ is odd (resp.~even).
It is clear that an odd Siegel modular form is identically zero
and that an odd Teichm\"uller modular form vanishes on the hyperelliptic locus.
Further, a Siegel modular form of weight $w(\rho)$ pulls back
to a Teichm\"uller modular form of the same weight.
Moreover, we have the following result.

\begin{lemma}
An even Teichm\"uller modular form
is the pullback of a Siegel modular form.
\end{lemma}
\begin{proof}
Note that the complement of the image of $\M{3}$ in $\A{3}$
is $\A{2,1}$, the locus of products of principally polarized abelian varieties,
of codimension two in $\A{3}$.
An even Teichm\"uller modular form of weight $w(\rho)$ 
is invariant under $\iota$ and descends to
a section of $\EE_{\rho}$ over $t(\M{3})$. Since $\A{3}$ is a smooth
stack and the codimension of $\A{2,1}$ is two, this section extends
to $\A{3}$ and yields the Siegel modular form that pulls back
to the Teichm\"uller modular form.
\end{proof}
We also have the following extension of Ichikawa's result.
\begin{lemma} \label{divisible}
An odd Teichm\"uller modular form
is divisible by $\chi_9$, i.e., it is
the product of $\chi_9$ and an even Teichm\"uller modular form.
\end{lemma}
\begin{proof}
Let $f$ be an odd Teichm\"uller modular form. Consider $g=f/\chi_9$,
a priori a meromorphic Teichm\"uller modular form with possibly a simple
pole along $\Hy3$.
Since $f$ vanishes on $\Hy3$, it follows that $g$ is a regular
Teichm\"uller modular form, which clearly is even. (So $g$ is the
pullback of a Siegel modular form.)
\end{proof}

A Teichm\"uller modular form extends holomorphically to $\barM{3}$~:

\begin{proposition}\label{extensiontoM3bar}
Let $f$ be a Teichm\"uller modular form of weight $w(\rho)$.
Then $f$ extends to a holomorphic section of 
${\EE}^{\prime}_{\rho}$ on $\barM{3}$. 
\end{proposition}
\begin{proof}
If $f$ is even, then it is the pullback of a Siegel modular form $F$.
By the Koecher principle, $F$ extends to $\tilde{\mathcal{A}}_3$.
The pullback of the extension gives the extension of $f$ to $\barM{3}$.
If $f$ is odd, it is
the product of $\chi_9$ and an even Teichm\"uller modular form.
Both extend, so $f$ extends as well.
\end{proof}

\begin{remark}
In the appendix (\S\ref{appendix}), we prove
that a Teichm\"uller modular form extends to $\barM{g}$ for
any $g\ge3$.
\end{remark}

\begin{corollary}
An odd Teichm\"uller modular form is a cusp form, 
i.e., it vanishes along $\delta_0$. It also vanishes
with multiplicity at least three along $\delta_1$. If it is
nonzero and
of weight $w(\rho)$, then $\rho_3\ge 9$.
\end{corollary}

Consider the modular form $\chi_{4,0,8}$.
We define a Siegel modular cusp form of weight $(4i,0,8i)$ by 
taking the $i$th symmetric power of $\chi_{4,0,8}$
and projecting:
$$ \chi^{(i)}= {\rm pr}_{[4i,0,8i]}(\Sym^i(\chi_{4,0,8})), $$
where ${\rm pr}_{[4i,0,8i]}$ is the projection
of the $i$th symmetric power of %${\EE}_{4,0,8}$ to ${\EE}_{4i,0,8i}$
$\Sym^4(\EE)\otimes\det^8(\EE)$ onto $\Sym^{4i}(\EE)\otimes\det^{8i}(\EE)$.

\begin{lemma}
The form $\chi^{(i)}$ is not identically zero.
\end{lemma}
\begin{proof}
We use that $\chi_{4,0,8}\otimes \Delta$ is obtained as the first term
of the degree $4$ scalar-valued Siegel cusp form $J_8$ of weight $8$ 
when developed in the normal directions of $\mathfrak{H}_3\times
\mathfrak{H}_1$. By developing $J_8^i$ we find $\chi^{(i)}\otimes \Delta^{i}$,
since the order of vanishing of $J_8^i$ along
$\mathfrak{H}_3\times \mathfrak{H}_1$ is $i$ times the order of  $J_8$.
\end{proof}
\smallskip
Note that the order of $\chi^{(i)}$ along $D$ is at least $i$. We now give 
a result on the orders of vanishing of $\chi_{18}$ and $\chi_{4,0,8}$.

\begin{lemma}
The orders of vanishing of $\chi_{18}$ and $\chi_{4,0,8}$  
along the hyperelliptic locus $H$, the divisor $D$
at infinity, and the locus $\mathcal{A}_{2,1}$ 
of products are given in the following table.
\end{lemma}
%\smallskip
%
\vbox{
\bigskip\centerline{\def\quad{\hskip 0.6em\relax}
\def\quod{\hskip 0.5em\relax }
\vbox{\offinterlineskip
\hrule
\halign{&\vrule#&\strut\quod\hfil#\quad\cr
height2pt&\omit&&\omit&&\omit&&\omit&\cr
%\noalign{\hrule}
& $F$ && ${\rm ord}_D$  && ${\rm ord}_{H}$ && ${\rm ord}_{\mathcal{A}_{2,1}}$ &\cr
\noalign{\hrule}
& $\chi_{18}$ && $2$  && $1$ && $6$  & \cr
& $\chi_{4,0,8}$ && $1$  &&  $0$ && $2$  & \cr
\noalign{\hrule}
} \hrule}
}}

\begin{proof}
The vanishing orders of $\chi_{18}$ along $D$ and $H$ were given in Section 3,
the order along $\mathcal{A}_{2,1}$ is given in \cite[Proposition~5.2]{C-vdG}. 
If the form  $\chi^{(2)}$ would vanish on the hyperelliptic locus, 
we could divide it by $\chi_{18}$, obtaining a holomorphic
Siegel modular form of weight $(8,0,-2)$, which has to be zero.
So $\chi^{(2)}$ does not vanish along $H$.
It follows that ${\rm ord}_{H}\chi_{4,0,8}=0$.
For the remaining entries, see \S\ref{chi408}, in particular
Lemma~\ref{I2notI3}.
\end{proof} 

\end{section}

%%%%%%%%%%%%%%%%%%%%%%%%%%%%%%%%%%%%%%%%
%%%%%%%%%%%%%%%%%%%%%%%%%%%%%%%%%%%%%%%%
% EIGHT
\begin{section}{Invariants and Concomitants of Ternary Quartics}\label{invariants}
Let $V$ be a $3$-dimensional vector space over ${\CC}$ 
generated by elements $x,y$ and $z$. 
We will denote the space of ternary quartics by
$\Sym^4(V)$ and we write a ternary quartic as 
$$
\begin{aligned}\label{normalization}
f=a_0\, x^4+ 4\, a_1\, x^3y+4\, a_2\, x^3z+6\, a_3\, x^2y^2+12\, a_4\, x^2yz
+6\, a_5\, x^2z^2+&\\
4\, a_6\, xy^3 +12\, a_7\, xy^2z+12\, a_8\, xyz^2+4\, a_9 \, xz^3 + a_{10}\, y^4+&\\
4\, a_{11} \, y^3z+  6\, a_{12}\, y^2z^2+4\, a_{13}\, yz^3+a_{14}z^4 & \, . \\
\end{aligned}
$$
Note that we order the monomials lexicographically and normalize $x^ay^bz^c$
by the factor $4!/a!\, b!\, c!$. 

An element $A=(a_{ij})$ of the  group ${\rm GL}(3,{\CC})$ 
acts on $\Sym^4(V)$ via
$$f(x,y,z) \mapsto 
f(a_{11}x+a_{21}y+a_{31}z,a_{12}x+a_{22}y+a_{32}z,a_{13}x+a_{23}y+a_{33}z) \, .
$$
We also consider the induced actions of
${\rm PGL}(3,{\CC})$ and ${\rm SL}(3,{\CC})$ on ${\mathcal Q}
={\PP}(\Sym^4(V))$. We take $(a_0,a_1,\ldots,a_{14})$ as coordinates
on $\Sym^4(V)$ and ${\mathcal Q}$. The natural ample line bundle ${\mathcal L}=
{\mathcal O}_{\mathcal Q}(1)$ on ${\mathcal Q}$ admits an action
of ${\rm SL}(3,{\CC})$ compatible with the action on the 
projectivized space ${\mathcal Q}$ of ternary quartics, 
cf.\ \cite{Mumford-GIT}.

By definition an invariant is an element of the ring
$$
I= \oplus_{n\geq 0} 
{\rm H}^0({\mathcal Q}, {\mathcal L}^n)^{{\rm SL}(3,{\CC})}\, ,
$$
that is, it is a polynomial in the coefficients 
$a_0,\ldots,a_{14}$ of $f$ 
which is invariant under the action of ${\rm SL}(3,{\CC})$.
By the work of G. Salmon, T. Shioda and J. Dixmier 
(see \cite{Salmon}, \cite{Shioda} and \cite{Dixmier}),
see also the unpublished work of T. Ohno \cite{Ohno},
we know the structure of the ring of invariants of ternary quartics 
(see  \cite[Theorem 3.2]{Dixmier}:
seven algebraically independent generators of degrees
$3, 6, 9, 12, 15, 18, 27$ which form a system of parameters of 
$I$ and six more basic invariants of degrees $9, 12, 15, 18, 21, 21$).
(See also \cite{Elsenhans} and the recent paper \cite{LerRit}.)
The Poincar\' e series of this
graded ring was determined by T. Shioda (see \cite{Shioda}):
\[
\sum_{n=0}^{\infty}
\dim (I_n)\, t^n=
\frac
{P(t)}
{(1-t^3)(1-t^6)(1-t^9)(1-t^{12})(1-t^{15})(1-t^{18})(1-t^{27})}\, ,
\]
where $I_n$ denotes the graded piece of $I$ of degree $n$ and $P$ is 
the palindromic polynomial
$$
\begin{aligned} 
t^{75}+t^{66}+t^{63}+t^{60}+2\,t^{57}+3\,t^{54}+2\,t^{51}
+3\,t^{48}+4\,t^{45}+3\,t^{42}+4\,t^{39}+ &\\
4\,t^{36}+3\,t^{33} 
+4\,t^{30}+3\,t^{27}+2\,t^{24}+3\,t^{21}+2\,t^{18}+
t^{15}+t^{12}+t^{9}+1& \, .
\end{aligned}
$$
This generating series  $\sum_{n=0}^{\infty} \dim (I_n) \, t^n$ 
starts as follows
\[
1+t^{3}+2\,t^{6}+4\,t^{9}+7\,t^{12}+11\,t^{15}+19\,t^{18}
+29\,t^{21}+44\,t^{24}+67\,t^{27}+\ldots
\]

Before we come to the notion of concomitant, we will 
fix coordinates on $\wedge^2 V$. We take
$$
\hat{x}=y\wedge z, \qquad \hat{y}=z \wedge x, \qquad \hat{z}=x \wedge y\, ; \eqno(8)
$$
then a basis of $\Sym^i(V) \otimes \Sym^j(\wedge^2 V)$
is given by the lexicographically ordered monomials
of degree $i$ in $x,y,z$ and degree $j$ in $\hat{x}, \hat{y},\hat{z}$.

We recall the notion of concomitants for ternary quartics.
Write $W=V^{\vee}$.
Consider
an equivariant
inclusion of $\SL{3}{\CC}$-representations
$$
\varphi\colon A \to \Sym^d (\Sym^4(W)),
$$
or equivalently
$$
{\CC} \hookrightarrow \Sym^d (\Sym^4(W))\otimes A^{\vee}\, ,
$$
where $\CC$ denotes the trivial representation; write $\Phi$ for the
image of $1$ under this map.
If $A=W[\rho]$ is the irreducible representation of highest weight
$\rho$ we call $\Phi$ a concomitant of type $(d,\rho)$.
It can be viewed as a form of degree $d$ in the coordinates $a_i$ of
the ternary quartic and of degree $\rho_1-\rho_2$ (resp.\
$\rho_2-\rho_3$) in $x,y,z$ (resp.\ $\hat{x},\hat{y},\hat{z}$).
In essence, we follow here
Chipalkatti's set-up in \cite{Chipalkatti}, but the notation is the dual one
and our basis of $\Sym^4(V)$ includes the multinomial coefficients.
Also, although we work with $\SL{3}{\CC}$ here, we have used
the notation introduced earlier
for the irreducible representations of $\GL{3}{\CC}$,
since we will soon need to work with the latter group.

The simplest nontrivial example is the case where $d=1$ and $A=\Sym^4(W)$
and the inclusion is the identity. 
Then the
corresponding concomitant $\Phi$ is the universal ternary quartic $f$ 
given in the beginning of this section. 

For $d=2$ we have the decomposition
$$
\Sym^2(\Sym^4(W))=W[8,0,0]+W[6,2,0]+W[4,4,0]\, .
$$
The concomitants corresponding to the first and last isotypical component are
$f^2$ and a classical concomitant denoted $\sigma$
by Salmon (see \cite[p.~264]{Salmon}). It is a quartic polynomial in 
$\hat{x},\hat{y},\hat{z}$ with coefficients of degree $2$ in the $a_i$, see 
Section \ref{constructing}.

The concomitants form a module $C$ over the ring $I$ of invariants.
\end{section}

\begin{section}{Quartic Curves and Moduli} \label{Quartics}
Let $\M{3}$ be the moduli space of curves of genus $3$ 
and ${\mathcal M}_3^{\rm nh}=
\M{3}-\mathcal{H}_{3}$ 
the open part 
of nonhyperelliptic curves.
The canonical image of a nonhyperelliptic curve is a quartic curve in
${\PP}^2$.  

To describe $\M{3}^{\rm nh}$ we consider a $3$-dimensional vector space
$V$ over a field $k$, say generated by $x,y,z$. 
We may view $\Sym^4(V)$ as the space of ternary quartics. 
There is a natural left
action of ${\rm GL}(V)$ on $\Sym^4(V)$: view $f\in\Sym^4(V)$
as a homogeneous polynomial map from the dual $V^{\vee}$ to $k$ of
degree $4$; 
the action of $A\in\mathrm{GL}(V)$ on $f$ is defined as
$$A\cdot f = f \circ A^t.$$
(Recall: for $B$ a linear map from $V_1$ to $V_2$, the transpose $B^t$
is the natural map from $V_2^{\vee}$ to $V_1^{\vee}$ obtained
by composition with $B$.)

In fact, for our purposes it is better to twist $\Sym^4(V)$ by the
inverse of the determinant, so we consider
$$Q=\Sym^4(V)\otimes{\det}^{-1}(V),$$
the irreducible representation of $\glv$ of highest weight
$(3,-1,-1)$. The point is that $c\cdot {\rm Id}_V$ acts on $Q$ as
$c\cdot {\rm Id}_Q$. Let $U\subset Q$ be the $\glv$-invariant open subset
corresponding to nonsingular projective plane quartics.

\begin{proposition}
There exists an isomorphism of algebraic stacks
$$
[U/\glv] \langepijl{\sim} {\mathcal M}_3^{\rm nh} \, .
$$
\end{proposition}
\begin{proof}
The standard construction of the coarse moduli space $M_3^{\rm nh}$ is as
the quotient of $\PP(U)$ by $\mathrm{PGL}(V)$. Since the embedding
is canonical, the stack quotient $[\PP(U)/\mathrm{PGL}(V)]$ gives
the stack ${\mathcal M}_3^{\rm nh}$. By the above, the stack quotient
of $U$ by the center of $\glv$ may be identified with $\PP(U)$. This gives
the result.
\end{proof}

Actually, we need an extension of this result to an open subset
of $Q$ with complement of codimension at least two. Let $U'\subset Q$
be the $\glv$-invariant open subset of quartics that are either
nonsingular or have one singularity, an ordinary double point.
Correspondingly, let $${\mathcal M}_3^{{\rm nh},*}
=\overline{\mathcal{M}}_3 - \overline{\mathcal{H}}_3
- \Delta_1 - \Delta_{00}$$
be the partial compactification of ${\mathcal M}_3^{\rm nh}$
consisting of nonhyperelliptic stable curves of genus $3$
with at most one node, which is nondisconnecting.
(Here $\Delta_1$ is the boundary component of curves
with a disconnecting node and $\Delta_{00}$ is the closure
of the boundary stratum of irreducible curves with exactly two
%nondisconnecting 
nodes.)
Essentially the same proof as above gives:

\begin{proposition}
There exists an isomorphism of algebraic stacks
$$
[U'/\glv] \langepijl{\sim} {\mathcal M}_3^{{\rm nh},*} \, .
$$
\end{proposition}

For $f\in U'$,
the elements $x,y,z$ give a basis of the space of 
global sections of the dualizing sheaf
on the quartic in ${\PP}^2$ defined by $f$.  
This globalizes and we obtain the following result.

\begin{corollary}\label{EvsV}
The pullback to $U'$ under the quotient map 
$q\colon U'\to {\mathcal M}_3^{{\rm nh},*}$
of the Hodge bundle $\EE^{\prime}$ is the
$\glv$-equivariant bundle $U'\times V$.
\end{corollary} 
Hence the pullback under $q$ of 
$\Sym^4(\EE^{\prime})\otimes\det^{-1}(\EE^{\prime})$ is
the $\glv$-equivariant bundle $U'\times Q$, which admits the diagonal 
section, the universal (at most $1$-nodal) ternary quartic.

\begin{corollary}\label{tauto}
The diagonal section of $U'\times Q$ descends to a canonical section
of $\Sym^4(\EE^{\prime})\otimes\det^{-1}(\EE^{\prime})$ over 
${\mathcal M}_3^{{\rm nh},*}$.
\end{corollary}

We revisit the notion of concomitants for ternary quartics.
Working with $\GL{3}{\CC}$ instead of $\SL{3}{\CC}$, it is best
to work with $Q$ or its dual $R$, as explained above.
We consider an equivariant
inclusion of $\GL{3}{\CC}$-representations
$$
\varphi\colon A \to \Sym^d (R),
$$
or equivalently
$$
{\CC} \hookrightarrow \Sym^d (R)\otimes A^{\vee}\, ,
$$
where $\CC$ denotes the trivial representation; write $\Phi$ for the
image of $1$ under this map.
If $A=W[\rho]$ is the irreducible representation of highest weight
$\rho$ we call $\Phi$ a concomitant of type $(d,\rho)$.
It can be viewed as a form of degree $d$ in the coordinates $a_i$ of the 
twisted
ternary quartic and of degree $\rho_1-\rho_2$ (resp.\
$\rho_2-\rho_3$, resp.~$\rho_3$) in $x,y,z$ (resp.\ $\hat{x},\hat{y},\hat{z}$,
resp.~$x\wedge y\wedge z$).
Note that $d=\rho_1+\rho_2+\rho_3$, so $d$ is determined by $\rho$
in the $\GL{3}{\CC}$-setting.
Sometimes, we will speak of concomitants of type $\rho$.

E.g.,
for $d=2$ we have the decomposition
$$
\Sym^2(R)=W[6,-2,-2]+W[4,0,-2]+W[2,2,-2]\, ;
$$
compared with the earlier decomposition of $\Sym^2(\Sym^4(W))$,
in each summand all entries have been lowered by $d=2$.

\end{section}
%%%%%%%%%%%%%%%%%%%%%%%%%%%%%%%%%%%%%%%
\begin{section}{Invariants, Concomitants and Modular Forms}\label{ICMF}
We consider a scalar-valued Teichm\"uller modular form $F$ of weight $k$, 
that is, a section of ${\det^k({\EE}')}$ 
on $\overline{\mathcal{M}}_3$, cf.~Proposition~\ref{extensiontoM3bar}.
By Corollary~\ref{EvsV},
the pullback of $\det^k({\EE}')$ to $U'$ under $q$
is $\det^k(V)$.
The pullback of $F$ is a section of $\det^k(V)$ over $U'$ that
extends to a section of $\det^k(V)$ over $Q$ (since the complement
of $U'$ has codimension two).
The corresponding irreducible representation of $\glv$ 
has highest weight $(k,k,k)$; it occurs in 
$\Sym^{3k}(Q)$. Thus $F$ defines
an invariant $\iota_F \in I$ of degree $3k$ (for ${\rm SL}(3,{\CC})$).

Conversely, if we have a homogeneous invariant,
necessarily of degree $3k$ with $k\in\NN$, 
it gives a section of $\det^{k}(V)$. It descends to
a holomorphic section of $\det^k(\EE')$ over the open
subset $\mathcal{M}_3^{\rm nh}$
that extends to a meromorphic section over $\M3$.
Since $\chi_9$ vanishes exactly once on the hyperelliptic locus,
we conclude that we get %inclusions
injections
$$
{T} \longrightarrow I \longrightarrow {T}_{\chi_{9}}\,  
\eqno(9)
$$
with ${T}$ the graded 
ring of scalar-valued Teichm\"uller modular forms and ${T}_{\chi_9}$
the ring obtained by inverting $\chi_9$.
Under the first map above, the modular form $\chi_9$ maps to (a nonzero multiple of) 
the discriminant, an invariant of degree $27$.
(In the recent paper \cite{LerRit}, 
Lercier and Ritzenthaler study the relation between
scalar-valued Siegel modular forms and invariants of ternary quartics
in detail. Their approach differs from ours.)

We now generalize this to concomitants for $\GL{3}{\CC}$. 
A concomitant $\Phi$ of type $(d, \rho)$ gives an equivariant section
of $V_{\rho}$.
It descends to a holomorphic section  of
${\EE}^{\prime}_{\rho}$
on ${\mathcal M}_3^{\rm nh}$ that extends to a meromorphic section
$\gamma(\Phi)$ on $\M3$.
After multiplication with the $r$th
power of  $\chi_9$ for $r$ large enough, it becomes
a holomorphic section of 
${\EE}^{'}_{\rho}\otimes \det^{9r}({\EE}')$
(as we shall see $r=d$ suffices).
Conversely,
if  $F$ is a Teichm\"uller modular form of weight $\rho$, 
that is, a section of ${\EE}^{\prime}_{\rho}$, then by pulling it back 
from ${\mathcal M}_3^{{\rm nh},*}$
to $U'$ we get an equivariant section of 
$V_{\rho}$ that extends to $Q$: a concomitant $\beta(F)$ of type $\rho$.

We thus find a commutative diagram 
$$
\begin{xy}
\xymatrix{
\Sigma \ar[r]^{\beta} & C \ar[r]^{\gamma} & \Sigma_{\chi_{9}} \\
{T} \ar[u] \ar[r] & I \ar[u] \ar[r] & {T}_{\chi_{9}}\ar[u] \\
}\end{xy} 
\eqno(10)
$$
where $\Sigma=\oplus_{\rho} T_{\rho}$ (with $\rho$ running over
the irreducible representations of ${\rm GL}(3,\CC)$)
is the ${T}$-module of vector-valued Teichm\"uller modular forms, 
%(with $\rho$ running over
%the irreducible representations of ${\rm GL}(3,\CC)$)
$\Sigma_{\chi_9}$ denotes the module obtained by inverting $\chi_9$
and $C$ is the $I$-module of concomitants.
A modular form $F$ of weight $(i,j,k)$ is sent to a concomitant of 
type 
$$
(d,\rho)=(i+2j+3k,[i+j+k,j+k,k])
$$ and a concomitant of
type $(d,\rho)$ to a (meromorphic) form of weight 
$$(\rho_1-\rho_2,\rho_2-\rho_3, \rho_3).$$
Note that the parity of the form equals the parity
of the degree $d$ of the concomitant.
 
A basic observation is the following.

\begin{proposition} \label{basic}
The image of the universal ternary quartic $f \in C$ under $\gamma$ equals 
up to a nonzero multiplicative constant 
the meromorphic Teichm\"uller modular form 
$$
\chi_{4,0,-1}=\chi_{4,0,8}/\chi_9\, .
$$
\end{proposition}
\begin{proof}
The tautological $f$ defines an element $\gamma(f)$ of $\Sigma_{\chi_9}$, a meromorphic
Teichm\"uller modular form  of weight $(4,0,-1)$. In order to identify it we
consider the Siegel (or Teichm\"uller) modular form $\chi_{4,0,8}$. 
By Proposition~\ref{Frobeniusphi0} we know that on the open set 
${\mathcal M}_3^{\rm nh}$ the two sections
$\gamma(f)$ and $\chi_{4,0,8}$ differ by a nowhere vanishing holomorphic
function. Therefore they differ by a power of $\chi_9$.
Since the weights are $(4,0,-1)$ and $(4,0,8)$, it follows that up to a
nonzero multiplicative constant we have $\gamma(f)=\chi_{4,0,8}/\chi_9$.
\end{proof}

Then the map $\gamma: C \to \Sigma_{\chi_{9}}$ can be written as
a substitution 
$$
c \mapsto c \circ \chi_{4,0,-1} 
$$
in the following sense.
Recall that we write our ternary quartic as $f=\sum_I n_I a_I x^I$ where
$I$ runs over the indices $(i_1,i_2,i_3)$ with $i_1+i_2+i_3=4$, 
$n_I= 4!/(i_1!i_2!i_3!)$ and 
$x^I=x^{i_1}y^{i_2}z^{i_3}$. 
We write $\chi_{4,0,8}$ in a similar way as a vector
$$
\chi_{4,0,8}= \sum_{I} n_I\alpha_I x^I\, . \eqno(11)
$$
Then each $\alpha_I$ is a holomorphic function on $\mathfrak{H}_3$ 
or also on Teichm\"uller space ${\mathcal T}_3$.
The map $\gamma$ is given by substituting 
in a concomitant
the meromorphic functions
$\alpha_I/\chi_9$ on ${\mathcal T}_3$
for the $a_I$.

Sometimes we prefer to work with holomorphic modular forms.
Then we don't use
the map $\gamma$, but a slightly adapted map $\gamma'$
that maps the tautological $f$ to $\gamma(f)\chi_9$ and that is defined by
 substituting in a concomitant $c$ of degree $d$ the 
$\alpha_I$ 
for the $a_I$. The result is a holomorphic vector-valued
modular form 
$$\gamma^{\prime}(c)=\gamma(c)\chi_9^d=c\circ \chi_{4,0,8}\,.
$$ 
For a concomitant of type $(d,\rho)$ it is a form of weight 
$(\rho_1-\rho_2,\rho_2-\rho_3,\rho_3+9d)$, i.e.,
a section of ${\EE}_{\rho} \otimes \det^{9d}{\EE}$.
We can calculate the Fourier expansion of 
$\gamma'(c)$ from the Fourier expansion of $\chi_{4,0,8}$.

Depending on the parity of $\rho$, the form $\gamma(c)$
is a (meromorphic) Siegel or Teichm\"uller modular form,
but $\gamma'(c)$ 
is (the pullback of) a holomorphic Siegel modular form.
If it vanishes along the
hyperelliptic locus, we can divide it by $\chi_{18}\,$. 
Note that the order of divisibility by $\chi_{18}$ 
is bounded above by $1/2$ of the order along $\delta_0$
and also by $1/6$ of the order along $\delta_1$.

\end{section}
%%%%%%%%%%%%%%%%%%%%%%%%%%%%%%%%%%%%%%%%
\begin{section}{The Order along the Locus of Double Conics}
Inside the space ${\PP}^{14}$ of ternary quartics there is the locus $DC$ 
of double conics. 
The order of a concomitant $c$ of ternary quartics along the locus $DC$
is determined as the order in the parameter $t\in\CC$ of the evaluation
of $c$ on the quartic $tf+g^2$, where $f$ is a general ternary quartic
and $g$ a general ternary conic.

\begin{proposition}	\label{orderDC}
Let $c$ be a concomitant of degree $d$ vanishing with order $v$ along
the locus of double conics. 
Then the order of $\gamma(c)$ along the hyperelliptic locus
$\Hy3\subset\M3$ equals $2v-d$. Hence the order
of $\gamma'(c)$ along $\Hy3$ equals $2v$ and the
corresponding Siegel modular form has order $v$ along $H\subset\A3$.
\end{proposition}
\begin{proof}
The proof is completely analogous to that of Theorem~1 in \cite{CFvdG2},
so we only mention the necessary changes.
The meromorphic form $\gamma(c)$ is obtained by substituting
$\chi_{4,0,-1}$ in $c$ and $\chi_{4,0,-1}$ has a simple pole along
$\Hy3$, so the result holds for the universal $f$ (with $d=1$ and $v=0$). 
After multiplying with $f$ or $f^2$ if necessary,
we may assume that $d$ is divisible by three, equal to $3e$.
We now let $A$ be the invariant of degree $3$:
$$
A=
a_0a_{10}a_{14}-4a_0a_{11}a_{13}+3a_0a_{12}^2-4a_1a_6a_{14}+{12}a_1a_7a_{13}-{12}a_1a_8a_{12}%+4a_1a_9a_{11}
+\dots\,\,.
%+4a_2a_6a_{13}-{12}a_2a_7a_{12}
%+{12}a_2a_8a_{11}-4a_2a_9a_{10}+3a_3^2a_{14}-{12}a_3a_4a_{13}+6a_3a_5a_{12}-{12}a_3a_7a_9+{12}a_3a_8^2+{12}a_4^2a_{12}-{12}a_4a_5a_{11}+{12}a_4a_6a_9-{12}a_4a_7a_8+3a_5^2a_{10}-{12}a_5a_6a_8+{12}a_5a_7^2
$$
One easily checks that $A$ doesn't vanish along $DC$, so $\gamma(A)$
has a pole of order $3$ along $\Hy3$. We write $\gamma(c)$
as $\gamma(c/A^e)\cdot \gamma(A)^e$, where $c/A^e$ is a meromorphic
concomitant of degree $0$. Its components are meromorphic functions
that descend to the components of $\gamma(c/A^e)$.
Now recall that the coarse moduli space $M_3$ may be constructed
by blowing up the locus of double conics in the projective space
of ternary quartics, deleting the proper transform of the
discriminant hypersurface, and taking the quotient by 
$\PGL3{\CC}$. The order of $\gamma(c/A^e)$ along $\Hy3$ equals
twice the order of $c/A^e$ along the exceptional divisor; this reflects
the difference between the stack $\M3$ and $M_3$. On the other hand,
the order along the exceptional divisor equals the order along $DC$.
This proves the first result. The other results are immediate
consequences.
\end{proof}

\begin{corollary}
Same hypothesis. If $d$ is even, then the order along $H$
of the meromorphic Siegel modular form corresponding
to $\gamma(c)$ equals $v-d/2$.
If $d$ is odd, then the order along $H$
of the meromorphic Siegel modular form corresponding
to $\chi_9\gamma(c)$ equals $v-(d-1)/2$.
\end{corollary}

\begin{corollary}
Let $F$ be a Teichm\"uller modular form of type $\rho$
that has order $m$ along $\Hy3$. Then $\beta(F)$ is
a concomitant of type $\rho$, so of degree
$d=\rho_1+\rho_2+\rho_3$, that has order $\tfrac12(m+d)$
along the locus of double conics. In particular, $m$ and
$d$ have the same parity (which equals that of $F$).
\end{corollary}
\begin{proof}
This follows from the proposition, since the composition
$\gamma\circ\beta$ of the maps in diagram (10) above
is the identity map onto $\Sigma\subset\Sigma_{\chi_9}$.
That the parities agree was already known: it is obvious
in the even case and follows in the odd case from the fact
that $\chi_9$ has multiplicity $1$ along $\Hy3$.
\end{proof}

\begin{example}
Let $\Xi$ be the discriminant of ternary quartics,
an invariant of degree $27$. Then $\gamma(\Xi)=\chi_9$
and $\chi_9\gamma(\Xi)=\chi_{18}$. We know that
$\chi_{18}$ has order $1$ along $H$ and find that
$v=14$, in agreement with a result of Aluffi-Cukierman \cite{A-C}.
\end{example}

\begin{notation}
Let $m\ge0$ be an integer. We denote by
$C_{d,\rho}(-m\, DC)$ the vector space of concomitants of type $(d,\rho)$
that have order $\geq m$ along the locus $DC$ of
double conics.

We denote by $S_{i,j,k}^m$ the vector space of Siegel modular forms
of type $(i,j,k)$ that have order $\geq m$ along the
boundary divisor $D$. So $S_{i,j,k}^0=M_{i,j,k}$ and
$S_{i,j,k}^1=S_{i,j,k}$.

These notations also make sense when $m<0$, but in most cases this doesn't give
anything new, since we consider regular or holomorphic concomitants
and modular forms here. However, see Corollary~\ref{chi18b} below.
\end{notation}
\begin{theorem} \label{the_isom}
Notation as above. There exists an isomorphism
$$
\varphi\colon C_{d,\rho}(-mDC) \buildrel\sim\over\longrightarrow 
S_{\rho_1-\rho_2, \rho_2-\rho_3,\rho_3+9n}^n \, ,
\qquad c \mapsto \gamma(c)\chi_9^n \, ,
$$
where $n=d-2m$.
\end{theorem}
\begin{proof}
Let $c\in C_{d,\rho}(-mDC)$.
By Proposition~\ref{orderDC}, the order along $\Hy3$ of $\gamma(c)$ is
at least $2m-d$. So $\varphi(c)$ is regular along $\Hy3$. Since $d$ and $n$
have the same parity, $\varphi(c)$ is (the pullback of) 
a Siegel modular form. It follows
immediately that $\varphi(c)\in S_{\rho_1-\rho_2, \rho_2-\rho_3,\rho_3+9n}^n
\,$. Now $\varphi$ is certainly injective, but also surjective: 
take $F\in S_{\rho_1-\rho_2, \rho_2-\rho_3,\rho_3+9n}^n\,$, 
then $\beta(F)/\Xi^n$ is a meromorphic concomitant that has order
at least $m$ along $DC$ and order at least $0$ along the locus
$\Xi=0$ of singular quartics. So it is in fact a regular concomitant
and an element of $C_{d,\rho}(-mDC)$.
\end{proof}

\begin{corollary}
If $i+2j+4k<36n$, then $S_{i,j,k}^n=0$.
\end{corollary}
\begin{proof}
Trivially, $C_{d,\rho}=0$ when $\rho_3<-d$. Since $i=\rho_1-\rho_2$, $j=\rho_2-\rho_3$,
and $k=\rho_3+9n$, we get $d=i+2j+3k-27n$ and obtain the result.
\end{proof}

We also obtain the following generalization
of the main result of~\S\ref{chi18}; this can also be deduced from the result
of Harris-Morrison \cite[Corollary~0.5]{Harris-Morrison}
on slopes of effective divisors on $\overline{\mathcal M}_3$.

\begin{corollary} \label{chi18b}
Let $k$ be a positive integer.
The space $S_{0,0,18k}^{2k}$ of cusp forms of weight $18k$ of order
at least $2k$ along $D$ is generated by $\chi_{18}^k$.
\end{corollary}
\begin{proof}
Take $d=0$, so $\rho=(0,0,0)$,
and $m=-k$, so $n=2k$. 
\end{proof}

We illustrate Theorem~\ref{the_isom} with a number of examples.

\begin{example} \label{bijv}

\begin{enumerate}

\item{}
Let $d=1$ and $m=0$. Then $\rho=(3,-1,-1)$ and $n=1$. 
We get an identification between the space of
concomitants of degree $1$, so of type $(3,-1,-1)$ and the space of cusp forms
$S_{4,0,8}$~: the universal ternary quartic~$f$ is mapped to $\chi_{4,0,8}$.
We also see that $S_{2,1,8}$, $S_{0,2,8}$, and $S_{1,0,9}$ all vanish
(which was known).
\item{}
Let $d=2$ and take $\rho=(4,0,-2)$ and $m=0$, so $n=2$.
We get an identification between the space of
concomitants of degree $2$ and type $(4,0,-2)$ and the space
$S_{4,2,16}^2=M_{4,2,16}(-2D)$ of Siegel modular forms vanishing at least
twice at the cusp. Hence the latter space is one-dimensional.
Analogously, $S_{8,0,16}^2$ and $S_{0,4,16}^2$ are also one-dimensional,
whereas
the vector spaces $S_{i,j,k+16}^2$ with $i+2j+3k=8$ are zero in all (seven)
other cases.
\item{} Let $d=3$ and take $\rho=(6,0,-3)$ and $m=1$, so $n=1$.
We get an isomorphism $C_{3,(6,0,-3)}(-DC)
\cong S_{6,3,6}$. The representation $W[6,0,-3]$ occurs with multiplicity
$1$ in $\Sym^3(R)$ and the associated concomitant vanishes
on the locus of double conics. Correspondingly, $\dim S_{6,3,6}=1$.
Completely analogously, $\rho=(4,1,-2)$ yields $\dim S_{3,3,7}=1$.

There are seven other nonzero concomitants of degree~$3$; they don't vanish
along~$DC$. Taking $m=0$, we find nine one-dimensional spaces
of cusp forms with order at least three along $D$
(two from the cusp forms just mentioned).
Finally, seventeen spaces $S_{i,j,k+6}$ 
and ten spaces $S^3_{i,j,k+24}$ with $i+2j+3k=12$ vanish.

\item{} Take $\rho=(5,4,-4)$, so $d=5$, and $m=2$, so $n=1$.
There is an isomorphism $C_{5,(10,9,1)}(-2DC)
\cong S_{1,8,5}$. The space $S_{1,8,5}$ is one-dimensional
and a generator of it has been constructed in the paper of Ibukiyama 
and Takemori \cite{I-T}. There is a unique concomitant of type $\rho$
and indeed it vanishes to order~$2$ along~$DC$.
We give more details in Section~\ref{degree5}.

\item{}
For $d=27$ and $\rho=(9,9,9)$, take $m=14$, so $n=-1$. On the right,
the space $S_{0,0,0}^{-1}$ is just $\CC$. Hence the space of invariants
of degree $27$ with order at least $14$ along $DC$ is one-dimensional;
it is generated by the discriminant $\Xi$, which is mapped to $1$,
since $\gamma(\Xi)=\chi_9$.

\end{enumerate}
\end{example}
\end{section}
%%%%%%%%%%%%%%%%%%%%%%%%%%%%%%%%%%%%%%%%
%%%%%%%%%%%%%%%%%%%%%%%%%%%%%%%%%%%%%%%%
\begin{section}{Constructing Modular Forms from Concomitants}\label{constructing}
\begin{subsection}{Degree $1$ and $2$}
In this section, we will use the $\SL3{\CC}$-notation
for concomitants as
in \S\ref{invariants}, mainly because it is easier to read. 
As long as the degree $d$ is known, this should not lead
to confusion. As mentioned
in \S\ref{Quartics}, the $\GL3{\CC}$-notation for a concomitant
of degree~$d$ is obtained by lowering the three entries by $d$.
 
There is one concomitant of degree $d=1$, the universal ternary quartic.
As we have seen, the image under the map $\gamma'$ of the basic 
concomitant $f$ is $\chi_{4,0,8}\,$.

For $d=2$ we have the decomposition
$$
\Sym^2(\Sym^4(W))=W[8,0,0]+W[6,2,0]+W[4,4,0]\, .
$$
The last isotypical component defines a covariant denoted 
$\sigma$ by Salmon (see \cite[p.~264]{Salmon}). 
It describes the curve in the dual ${\PP}^2$
of lines intersecting the curve defined by $f$ equianharmonically.
Its first terms are
%It can be 
given by (using $u_0=\hat{x}, u_1=\hat{y}, u_2=\hat{z}$)
% 
%\begin{footnotesize}
\begin{align*}
%C=&
%\sigma=&
&\sigma=
(a_{10}a_{14}-4a_{11}a_{13}+3a_{12}^2)u_0^4+
(4a_9a_{11}-12a_8a_{12}+12a_7a_{13}-4a_6a_{14})u_0^3u_1\\
&\qquad+(-4a_9a_{10}+12a_8a_{11}-12a_7a_{12}+4a_6a_{13})u_0^3u_2\\
&\qquad+(6a_5a_{12}-12a_4a_{13}+6a_3a_{14}-12a_7a_9+12a_8^2)u_0^2u_1^2+\dots\\
%&+(-12a_5a_{11}+24a_4a_{12}-12a_3a_{13}+12a_6a_9-12a_7a_8)u_0^2u_1u_2\\
%&+(6a_5a_{10}-12a_4a_{11}+6a_3a_{12}-12a_6a_8+12a_7^2)u_0^2u_2^2\\
%&+(-4a_1a_{14}+4a_2a_{13}+12a_4a_9-12a_5a_8)u_0u_1^3\\
%&+(12a_1a_{13}-12a_2a_{12}-12a_3a_9-12a_4a_8+24a_5a_7)u_0u_1^2u_2\\
%&+(-12a_1a_{12}+12a_2a_{11}+24a_3a_8-12a_4a_7-12a_5a_6)u_0u_1u_2^2\\
%&+(4a_1a_{11}-4a_2a_{10}-12a_3a_7+12a_4a_6)u_0u_2^3\\
%&+(a_0a_{14}-4a_2a_9+3a_5^2)u_1^4+(-4a_0a_{13}+4a_1a_9+12a_2a_8-12a_4a_5)u_2u_1^3\\
%&+(6a_0a_{12}-12a_1a_8-12a_2a_7+6a_3a_5+12a_4^2)u_2^2u_1^2\\
%&+(-4a_0a_{11}+12a_1a_7+4a_2a_6-12a_3a_4)u_2^3u_1
%+(a_0a_{10}-4a_1a_6+3a_3^2)u_2^4 \\
\end{align*}
%\end{footnotesize}
Under the map $\gamma'$  the three concomitants 
corresponding to the three terms in
the decomposition of $\Sym^2(\Sym^4(W))$
give rise to Siegel 
modular forms of weights $\chi_{8,0,16}$,
$\chi_{4,2,16}$ and $\chi_{0,4,16}$ respectively. The form $\chi_{8,0,16}$
is the (symmetric) square of $\chi_{4,0,8}$. 

Since $\chi_{4,0,8}$ vanishes with multiplicity
$1$ along $D$,  these three modular forms vanish with multiplicity $\geq 2$ along $D$.
In order to give the beginning of the Fourier expansion  
of $\chi_{0,4,16}\,$, we notice that
it suffices to give the coordinates $v_1,v_2,v_4$ and $v_5\,$, since 
under the action of $\mathfrak{S}_3$ given in (3)
the $15$ coordinates satisfy the same relations as
those of $\chi_{4,0,8}$. We have for the entries of the coefficient 
of $q_1^2q_2^2q_3^2$~: 
$$
v_1= 3\, (u-1)^4(v-1)^4/u^2v^2, \qquad  
v_2= 12\,  (u-1)^3(v-1)^3(w-1)(uvw+u-v-w)/u^2v^2w,
$$
$$
v_4= 6\, (u-1)^2(v-1)^2(w-1)^2 \, 
(2\sigma_3^2-4\sigma_3\sigma_1-2\sigma_3+2\sigma_1^2-8\sigma_2+9(u^2+1)vw)
/\sigma_3^2
$$
and
$$
v_5= -12\, (u-1)^2(v-1)^2(w^2-1) \, (\sigma_3^2-2\sigma_3\sigma_1+8\sigma_3+\sigma_1^2-4\sigma_2)/\sigma_3^2\, ,
$$
where as before $\sigma_i$ is the elementary symmetric function
of degree $i$ in $u,v,w$.

Summing up,
from Theorem~\ref{the_isom} and Example~\ref{bijv},
we have isomorphisms 
$$
C_{2,(8,0,0)} \cong S_{8,0,16}^2\,, \qquad
C_{2,(6,2,0)} \cong S_{4,2,16}^2\,, \qquad
C_{2,(4,4,0)} \cong S_{0,4,16}^2\,, \eqno(12)	\label{cased=2}
$$
given by $c \longleftrightarrow (c\circ \chi_{4,0,-1}) \chi_{18}$.

Note that the dimensions of $S_{8,0,16}$ (resp.\ $S_{4,2,16}$, $S_{0,4,16}$)
are $26$ (resp.\ $25$, $12$) (cf.~\cite{Taibi}).
\end{subsection}

\begin{subsection}{Degree $3$}
For $d=3$ we have the multiplicity-free decomposition
$$
\begin{aligned}
\Sym^3(\Sym^4(W))=W[12,0,0]+W[10,2,0]+W[9,3,0]+W[8,4,0]+W[8,2,2]+&\\
W[7,4,1]+W[6,6,0]+W[6,4,2]+W[4,4,4]\, . &\\
\end{aligned}
$$
The covariant that corresponds to $W[12,0,0]$ is given by the form
of degree $12$ that is the third power of $f$, while $W[8,2,2]$
corresponds to the covariant given by the Hessian of $f$. The
contravariant that corresponds to $W[6,6,0]$ is given by the dual
sextic of lines intersecting $f$ in a $4$-tuple with $j$-invariant $1728$.
Finally, $W[4,4,4]$ corresponds to the invariant of degree~$3$.

By the map $c \mapsto (c \circ \chi_{4,0,-1})\chi_9^3$ these concomitants 
yield modular forms of weights 
$$
(12,0,24), (8,2,24), (6,3,24), (4,4,24), (6,0,26), (3,3,25), (0,6,24),
(2,2,26), (0,0,28).
$$
In fact, we have isomorphisms between these spaces of concomitants and
spaces of cusp forms vanishing with multiplicity $\geq 3$ along $D$:
$$
C_{3,(a,b,c)} \cong S_{a-b,b-c,c+24}^3 \qquad {\rm via} \qquad
c \mapsto (c \circ \chi_{4,0,-1})\chi_9^3\, ,
$$
and also
$$
C_{3,(a,b,c)}(-DC) \cong S_{a-b,b-c,c+6} \qquad {\rm via} \qquad
c \mapsto (c \circ \chi_{4,0,-1})\chi_9\, .
$$
Only for $\rho=[9,3,0]$ and $[7,4,1]$ we get nonzero spaces
$S_{a-b,b-c,c+6}$~; namely, $\dim S_{6,3,6}=\dim S_{3,3,7}=1$, 
in perfect agreement with the fact that the
concomitants of degree~$3$ that vanish on the locus of double conics are those
corresponding to $[9,3,0]$ and $[7,4,1]$.

\bigskip

As an example, 
we take a closer look at the case of the concomitant $c$ provided by $W[7,4,1]$
in $\Sym^3(\Sym^4(W))$. The modular form $\gamma'(c)$ is divisible
by $\chi_{18}$ and yields a cusp form $\chi_{3,3,7}=\sum a(N) q^N$ 
of weight $(3,3,7)$ on $\Gamma_3$.
We can calculate some of its Fourier coefficients: we find the irreducible
representation of highest weight $(13,10,7)$ inside $\Sym^{3}(W)\otimes
\Sym^3(\wedge^2(W))\otimes {\det}^7(W)$ with $W$ 
the standard representation of ${\rm GL}(3)$. For example, using the
shorthand $N=[n_{11},n_{22},n_{33};2n_{12},2n_{13},2n_{23}]$ 
for the half-integral matrix $N=(n_{ij})$, we have
for $N_1=[1,1,1;1,1,1]$ 
$$
a(N_1)=[6,-20,-20,0,40,0,0,0,0,0,0,-27,-30,15,90, \ldots]^t
$$
and for $a(2N_1)$ we find
$$
8\, [1050,-2380,-2380,720,3320,720,0,0,0,0,1560,-4725,-2490,-15,1430,\ldots]^t
$$
and for $a([3,2,2,4,4,2])$ we find the vector
$$
4 [ -480, 860, 860, -430, -1060, -430, 0, 315, 315, 0, -1140, 2160, 2040, -1275, -2610, \ldots ]^t
$$
and then the formula for the Hecke operator for $(i,j,k)=(3,3,7)$
%\begin{tiny}
\begin{align*}
a_2(N_1)=& \, 
a(2\, N_1)
+
2^{i+2j+k-3}
\Sym^{-i}
\left(
\begin{smallmatrix}
1 & 1 & 1  \\
0 & 2 & 0  \\
0 & 0 & 2 
\end{smallmatrix}
\right)
\Sym^{-j}(\wedge^2(
\left(
\begin{smallmatrix}
1 & 1 & 1  \\
0 & 2 & 0  \\
0 & 0 & 2 
\end{smallmatrix}
\right))
a([3\, 2\, 2\,; 4\, 4\, 2])
\end{align*}
%\end{tiny}
gives the eigenvalue $\lambda_2=1080$.  The eigenvalues can be checked against the data provided by \cite{BFvdG3}.
\smallskip

\noindent
The invariant $\iota$ of type $(3,(4,4,4))$ is given by
$$
\begin{aligned}
a_{0}\, a_{10}\, a_{14}
-4\, a_{0}\, a_{11}\, a_{13}
+3\,a_{0}\, a_{12}^2+
4\, a_{1}\, a_{11}\, a_{9}
-12\, a_{1}\, a_{12}\, a_{8}
+12\, a_{1}\, a_{13}\, a_{7}
-4\, a_{1}\,a_{14}\, a_{6}&\\
-4\, a_{10}\, a_{2}\, a_{9}
+3\, a_{10}\,a_{5}^2
+12\, a_{11}\, a_{2}\, a_{8}
-12\, a_{11}\, a_{4}\, a_{5}
-12\, a_{12}\, a_{2}\, a_{7}
+6\, a_{12}\, a_{3}\,a_{5}
+12\, a_{12}\, a_{4}^2&\\
+4\, a_{13}\, a_{2}\, a_{6}
-12\, a_{13}\, a_{3}\, a_{4}
+3\, a_{14}\, a_{3}^2
-12\,a_{3}\, a_{7}\, a_{9}
+12\, a_{3}\, a_{8}^2
+12\, a_{4}\, a_{6}\, a_{9}
-12\, a_{4}\, a_{7}\, a_{8}&\\
-12\, a_{5}\,a_{6}\, a_{8}
+12\, a_{5}\, a_{7}^2&\\
\end{aligned}
$$
in the coefficients $a_i$ of the ternary quartic $f$.
It defines a cusp form $\chi_{28}$ of weight $28$ vanishing
with multiplicity $3$ at $D$. Using the pairing 
induced by the pairing of $W$ and $\wedge^2 W$, we get
$\langle f, \sigma\rangle =\iota$ or in terms of modular forms
$$
\langle \chi_{4,0,8}, \chi_{0,4,16} \rangle = \chi_{28}\, ,
$$
where the pairing of $\chi_{4,0,8}=\sum_I n_I \alpha_I x^I$ and
$\chi_{0,4,16}=\sum_I \beta_I \hat{x}^I$ is $\sum_I n_I \alpha_I \beta_I$.
The Fourier expansion of $\chi_{28}$ starts with
$$
F(\tau)=(\frac{(u-1)^2(v-1)^2(w-1)^2}{144\,u^3v^3w^3}
(u^4v^4w^4+u^4v^4w^3+\ldots))
q_1^3q_2^3q_3^3+\ldots \, .
$$
This Siegel modular form vanishes with order~$3$ at infinity and 
along ${\mathcal A}_{2,1}$ with order~$8$. Indeed, expanding $F$ as a Taylor series
along $\mathfrak{H}_2 \times \mathfrak{H}_1$ as done in \cite[Proposition~2.1]{C-vdG}
we get as first term a tensor product
$F' \otimes F^{\prime\prime}\in S_{8,28}(\Gamma_2)\otimes S_{36}(\Gamma_1)$ because
$F^{\prime\prime}$ will be a cusp form on $\Gamma_1$
vanishing with multiplicity $3$ at the cusp
and the first such form is $\Delta^3$. 
By looking around $\mathfrak{H}_1^3$ we see that
$F^{\prime}$ vanishes along $\mathfrak{H}_1^2$ with multiplicity $4$;
dividing $F^{\prime}$ by $\chi_{10}^2$ we get a cusp form of weight $(8,8)$ on $\Gamma_2$.
Now $S_{8,8}(\Gamma_2)$ has dimension $1$ and is generated by the form $\chi_{8,8}$
(see \cite[p.\ 11]{CFvdG}) and we find $F^{\prime}= (1/104230)^2 \chi_{10}^2 \chi_{8,8}\,$.

We thus see that the modular form $F$, viewed as Teichm\"uller form,
vanishes with order $3$
along $\delta_0$ and order $8$ along~$\delta_1$. It can be seen as a section of
${\mathcal O}(28\lambda-3\, \delta_0-8 \, \delta_1)$;
see the last remark of \cite[p.~1766]{OS}.
\end{subsection}

\begin{subsection}{Degree $5$}\label{degree5}
We have the isotypical decomposition of
$\Sym^5(\Sym^4(W))$ as 
\begin{tiny}
$$ 
\begin{aligned}
2\, W[8,6,6] + 2\, W[8,8,4] + W[9,6,5] + W[9,7,4] + W[9,8,3] + 4\, W[10,6,4]
+ 2\, W[10,7,3] &\\
+ 3\, W[10,8,2] + W[10,9,1] + W[10,10,0] + W[11,5,4] + 3\, W[11,6,3]
 + 2\, W[11,7,2] + W[11,8,1] &\\
+ 3\, W[12,4,4] + W[12,5,3] + 4\, W[12,6,2] + 
W[12,7,1] + 2\, W[12,8,0] + 2\, W[13,4,3] + 2\, W[13,5,2]&\\
 + 2\, W[13,6,1] + W[13,7,0] + 3\, W[14,4,2] 
+ W[14,5,1] + 2\, W[14,6,0] + W[15,3,2] + W[15,4,1] &\\
 + W[15,5,0] + W[16,2,2] + 2\, W[16,4,0] + W[17,3,0] + W[18,2,0] + W[20,0,0]. 
&\\
\end{aligned}
$$
\end{tiny}
The concomitant $c$ provided by $W[10,9,1]$ occurring in 
$\Sym^5(\Sym^4(W))$ vanishes with order 2 along the locus of double conics,
so the modular form $\gamma'(c)$ is divisible by $\chi_{18}^2$ and 
yields a cusp form $\chi_{1,8,5} \in S_{1,8,5}$. 
We calculate a few Fourier coefficients. We represent these
inside the representation 
$W\otimes \Sym^8(\wedge^2(W))\otimes {\det}^5(W)$
with $W$ the standard representation of ${\rm GL}(3)$. 
With  $N_1 = [1,1,1;1,1,1]$
we have
$$
a(N_1)=1536 \, [0, 0, 0, 0, 0, 0, 0, -1, 1, 0, 0, 2, 0, -2, 0, 0, -1, -2, \ldots]^t
$$
and $a(2N_1)$ is given by
$$
24\, [0, -4, 4, 8, 0, -8, -4, 17, -17, 4, 0, -50, 0, 50, 0, 4, 25, 50,\ldots]^t
$$
while $a([3,2,2;4,4,2])$ is given by
$$
192 \,
[0, 8, -8, -36, 0, 36, 68, 74, -74, -68, -70, -160, 0, 160, 70, 42, 155, 110,\ldots]^t.
$$
This gives the  Hecke eigenvalue at $p=2$: 
$\lambda_2=-2880=-24(216+2^2(-24))$, in agreement with
the fact that this cusp form is predicted to be a lift from $\Gamma_1$ with
 Hecke eigenvalues of the shape $\lambda_p=\tau(p)(b(p)+p^2\tau(p))$,
where $\tau(p)$ is the Fourier coefficient of $\Delta$ at $p$, while $b(p)$ is 
that of the unique normalized cusp form of weight 16 on $\Gamma_1$.
This modular form is also given  in \cite[\S 5.6]{I-T},
where it is constructed using theta functions.

In the case of the components 
\begin{tiny}
$$
W[10,6,4],\quad W[10,8,2],\quad W[11,6,3],\quad W[12,4,4],\quad
%$$ $$ 
W[12,6,2],\quad W[13,4,3],\quad W[13,6,1],\quad W[14,4,2]
$$
\end{tiny}
which occur with multiplicity $\geq 2$ one can find a nonzero concomitant $c$ 
vanishing with multiplicity $\geq 2$ on the locus of double conics, and then
$\gamma^{\prime}(c)/\chi_{18}^2$ defines a holomorphic cusp form
of weight
$$
(4,2,8),\quad (2,6,6),\quad (5,3,7),\quad (8,0,8),\quad (6,4,6),\quad (9,1,7),\quad (7,5,5),\quad (10,2,6)
$$
respectively,
and in these cases we checked that the eigenvalue for the Hecke
operator at $p=2$ agrees with the data given in \cite{BFvdG3}.
Also the case $W[8,6,6]$ gives a concomitant $c$ vanishing with
order $\geq 2$ on the locus of double conics. 
Then $\gamma^{\prime}(c)/\chi_{18}^2$ yields a cusp form in 
$S_{2,0,10}$. 
Its Fourier
expansion starts with
$$
\frac{1}{2308\, uvw}
\left( \begin{matrix} c_1 \\ \vdots \\ c_6 \end{matrix} \right)
q_1q_2q_3 +\cdots 
$$
with 
%\begin{tiny}
$$
\begin{aligned}
c_1=u^2v^2w^2+u^2v^2w+u^2vw^2+uv^2w^2-6\, u^2vw+-6\, uv^2w+ 14\, uvw^2 +u^2v+u^2w+&\\
uv^2
-20\, uvw+uw^2+v^2w+vw^2+u^2+14\, uv-6\, uw+v^2-6\, vw+w^2+u+v+w &\\
\end{aligned}
$$
%\end{tiny}
and
$$
c_2=u^2\, v^2w^2+u^2v^2w+u^2vw^2-6\, u^2vw+u^2v+u^2w-v^2w-vw^2+u^2-v^2+6\, vw-w^2-v-w
$$
and this determines the other coordinates by Lemma~\ref{actionS3}.
\end{subsection}
%%%%%%%%%%%%%%%%%
\begin{subsection}{Degree $6$}
The so-called catalecticant is an invariant of degree $d=6$ and is 
associated to $W[8,8,8]$ occurring in $\Sym^6(\Sym^4(W))$. 
It is given as
$$
i_6=
\left \vert
\begin{matrix}
a_0 & a_3 & a_5 & a_4 & a_2 & a_1\\
a_3 & a_{10} & a_{12} & a_{11} & a_7 & a_6 \\
a_5 & a_{12} & a_{14} & a_{13} & a_9 & a_8 \\
a_4 & a_{11} & a_{13} & a_{12} & a_8  & a_7 \\
a_2 &  a_7 & a_9 &  a_8 & a_5 & a_4 \\ 
a_1 & a_6 & a_8 & a_7 & a_4 &  a_3
\end{matrix}
\right \vert
$$
and gives rise to a Siegel modular form of weight $56$ vanishing with order $6$ at $D$
and with order at least $16$ along ${\mathcal A}_{2,1}$ since the first cusp form vanishing with
order $6$ at $\infty$ on $\Gamma_1$ is $\Delta^6$ of weight $72$ and $16=72-56$.
This modular form can be interpreted as a section of
${\mathcal O}(56\, \lambda-6\, \delta_0-16\, \delta_1)$ on $\overline{\mathcal M}_3$,
in agreement with \cite[p.\ 1766]{OS}.
Another description of this form can be found in \cite[Proposition~4.5]{LerRit}

\end{subsection}
\begin{remark}
Some of the modular forms constructed above have a geometric meaning. 
Chipalkatti proves in \cite{Chipalkatti} 
that a ternary quartic $f$ is the sum of the fourth powers 
of $s$ linear forms $f=\ell_1^4+\ell_2^4+\cdots +\ell_s^4$ for $1\leq s \leq 5$
if and only if the concomitants in a certain set $U_s$ vanish, where the set
$U_s$ is given by
$$
\begin{aligned}
&U_1=\{ c(2;4,2), c(2;0,4)\}, \qquad
U_2=\{ c(3;6,0), c(3;0,6), c(3;3,3), c(3,2,2), c(3;0,0)\}, \\
&U_3=\{ c(4;4,0), c(4;2,4), c(4;1,3), c(4;0,2) \},\\
&U_4=\{ c_1(5;0,4)-c_2(5;0,4), c(5;2,0) \},\qquad
U_5=\{ 3c(6,0,0)-c(3;0,0)^2\}; \\
\end{aligned}
$$
the concomitant $c(d;m,n)$ corresponds to the irreducible representation 
$W[m_1,m_2,m_3]$ occurring in $\Sym^d(\Sym^4(W))$ with $m=m_1-m_2$ and 
$n=m_2-m_3$ (and $4d=m_1+m_2+m_3$), see \cite[Theorem 4.1]{Chipalkatti}.
For example, the vanishing of the modular forms $\chi_{4,2,16}$ and $\chi_{0,4,16}$
signalizes this property for $s=1$.
\end{remark}
\end{section}
%%%%%%%%%%%%%%%%%%%%%%%%%%%%%%%%%%%%%%%%
\begin{section}{Teichm\"uller modular forms and the cohomology of 
local systems} It is well-known that Siegel modular forms of degree 
$g$ occur in the cohomology of local systems on the moduli space 
$\A{g}$ of principally polarized abelian varieties
of dimension~$g$.
Denoting by $\pi: {\mathcal X}_g \to {\mathcal A}_g$ 
the universal abelian variety,
we let $\VV=R^1\pi_* {\QQ}_{\ell}$ be the standard local system
of rank $2g$ on $\A{g}$. This comes with a symplectic pairing
${\VV} \times {\VV} \to {\QQ}_{\ell}(-1)$. 
For every irreducible representation
of the symplectic group ${\rm GSp}(2g,{\QQ})$ 
with highest weight $\mu$, we have a local
system $\VVl$ obtained from $\VV$ by applying a Schur functor.
We consider the `motivic' Euler characteristic 
$$
e_c(\A{g};{\VVl})=\sum_{i=0}^{g(g+1)}(-1)^i [H^i_c(\A{g},\VVl)]
$$
of compactly supported cohomology. 
The cohomology group $H^i_c(\A{g}\otimes {\CC},
\VVl\otimes {\CC})$ 
(resp.\ $H^i(\A{g}\otimes {\CC},\VVl\otimes {\CC})$) is
provided with a mixed Hodge structure of weights 
$\leq |\mu|+i$ (resp.\ $\geq |\mu|+i$)
and the sums of the elements of the $2^g$ subsets of
$\{ \mu_g+1,\mu_{g-1}+2,\ldots, \mu_1+g\}$
yield the degrees at which nontrivial steps in the Hodge filtration
may occur, see \cite{F-C} or
\cite{BFvdG2} and references there.
So the last step 
is $F^{|\mu|+g(g+1)/2}$ and it is here
that we find Siegel modular forms: there is an isomorphism
$$
S_{n(\mu)}= H^0(\barA{g}\otimes {\CC}, 
{\EE}_{\mu}\otimes {\det}^{g+1}({\EE})(-D))
\cong F^{|\mu|+g(g+1)/2} 
H^{g(g+1)/2}_c(\A{g}\otimes {\CC}, {\VVl}\otimes {\CC})
\eqno(13)
$$
where 
$$
n(\mu)=(\mu_1-\mu_2,\mu_2-\mu_3,\ldots, \mu_{g-1}-\mu_g,
\mu_g +g+1)\, .
$$ 
We denote by $H^i_{!}(\A{g},\VVl)$ the image of $H^i_c(\A{g},\VVl)\to
H^i(\A{g},\VVl)$. It is known that if $\mu$ is regular then 
$H_{!}^i(\A{g},\VVl)=(0)$ if $i\neq g(g+1)/2$. 
The above results are due to Faltings \cite{Faltings} and
Faltings-Chai \cite{F-C};
a key role is played by the (dual) BGG-complex.

We are interested in a similar interpretation of 
Teichm\"uller modular forms for $g\ge2$. 	
So far, we have only considered these for $g=3$. With a
Teichm\"uller modular form of type $\rho$ (or weight $w(\rho)$),
we mean here a section over $\barM{g}$ of $\EE^{\prime}_{\rho}$,
the bundle obtained by applying the Schur functor associated
to an irreducible representation $\rho$ of ${\rm GL}(g)$
to the Hodge bundle $\EE^{\prime}$.

By pulling back under the Torelli map $t$, 
we obtain local systems
$\VVl^{\prime}=t^*{\VVl}$ on $\M{g}$ and $\Mct{g}=\barM{g}-\Delta_0$,
the moduli space of curves of compact type.
We consider the motivic Euler characteristic
$$
e_c(\M{g};\VVl^{\prime})=
\sum_{i=0}^{6g-6}(-1)^i [H^i_c(\M{g},\VVl^{\prime})]
$$
and similarly we can consider $e_c(\Mct{g};\VVl^{\prime})$.

We obtain the following partial analogue of the above results.

\begin{theorem} \label{topHodge}
For the natural mixed Hodge structures
on the middle cohomology groups, we have the following isomorphisms:
$$ F^mH^{3g-3}(\M{g},\VlaC)\cong H^0(\barM{g},\EE^{\prime}_{\mu}\otimes
\mathcal{O}(13\lambda-\delta)), \eqno(14) $$
$$ F^mH^{3g-3}(\Mct{g},\VlaC)\cong H^0(\barM{g},\EE^{\prime}_{\mu}\otimes
\mathcal{O}(13\lambda+\delta_0-2\delta)), \qquad\text{\rm and} \eqno(15) $$
$$ F^mH^{3g-3}_c(\M{g},\VlaC) \cong F^mH^{3g-3}_c(\Mct{g},\VlaC)
\cong H^0(\barM{g},\EE^{\prime}_{\mu}\otimes \mathcal{O}(13\lambda-2\delta)),\eqno(16)$$
where $m=3g-3+|\mu|$ is the maximum possible Hodge degree.
\end{theorem}

\begin{proof}
It is not clear whether the BGG-complex can be adapted to this setting
and we resort to the logarithmic de Rham-complex.
The local system 
$\VlaC$ on $\M{g}$ (or $\Mct{g}$)
corresponds to a holomorphic vector bundle $\Vbla$ with
flat connection $\nabla$ (the Gauss-Manin connection). 
The boundary $D=\barM{g}-\M{g}$ is a divisor with normal crossings. 
The bundle $\Vbla$ admits a canonical extension $\Vbar$ to
$\barM{g}$ and $\nabla$ extends to a map
$$\nabla\colon\Vbar\to\Vbar\otimes\Omega^1_{\barM{g}}(\log D)$$
with nilpotent polar part.
Then the hypercohomology of the logarithmic de Rham-complex
computes the cohomology of $\VlaC$ on $\M{g}$~:
$$ H^p(\M{g}, \VlaC)
\cong \HH^p(\barM{g},\Vbar\otimes\Omega^{\bullet}_{\barM{g}}(\log D)).$$
Cf.~\cite{Schnell}.
Twisting with $\mathcal{O}(-D)$, we obtain a complex that computes
the compactly supported cohomology:
$$ H^p_c(\M{g},\VlaC)\cong \HH^p(\barM{g},
\Vbar\otimes\Omega^{\bullet}_{\barM{g}}(\log D)\otimes \mathcal{O}(-D)).$$
Analogous statements hold for $\Mct{g}\,$, after replacing $D$ by $\Delta_0$. 

These complexes admit natural Hodge filtrations,
which induce the Hodge filtrations of the mixed Hodge structures
on $H^p_{(c)}(\M{g},\VlaC)$ and $H^p_{(c)}(\Mct{g},\VlaC)$.
Let $m=3g-3+|\mu|$ be the top Hodge degree. The complex
$F^m(\Vbar\otimes\Omega^{\bullet}_{\barM{g}}(\log D))$
consists of the sheaf 
$\EE^{\prime}_{\mu}\otimes\Omega^{3g-3}_{\barM{g}}(\log D)$,
considered as a complex supported in degree $3g-3$.
As is well-known~\cite{HM}, the canonical bundle of the stack
$\barM{g}$ equals $\mathcal{O}(13\lambda-2\delta)$.
Since the spectral sequence associated to the Hodge filtration
degenerates at $E_1$, we obtain the stated isomorphisms.
\end{proof}

Thus the final steps in the Hodge filtrations on the middle
cohomology groups are isomorphic to spaces of Teichm\"uller
modular forms of highest weight $\mu+(13,13,\dots,13)$
with prescribed vanishing behaviour along
components of the boundary. This is entirely analogous to 
the relation between Siegel modular forms and the cohomology of $\A{g}\,$.
In particular, we have in genus $2$ the isomorphism
$\Mct{2}\cong\A{2}$ and the relation $10\lambda\sim\delta_0+2\delta_1$;
the isomorphisms for $\Mct{2}$ above agree with those for $\A{2}$
discussed earlier.
However, note that the Teichm\"uller modular forms found here are
cusp forms in a strong sense: they need to vanish (at least) 
once or twice along each component of the boundary.

The isomorphisms continue to hold if we change the cohomological
degrees on both sides by the same amount; in particular, $F^mH^p_{(c)}=(0)$
for $p<3g-3$. As a trivial example, take $\mu=0$, so $m=3g-3$; the
top compactly supported cohomology of $\M{g}$ and $\Mct{g}$ 
is spanned by $L^{3g-3}$ and coincides with $F^{3g-3}$; on the other hand,
$H^{3g-3}(\barM{g},\mathcal{O}(13\lambda-2\delta))\cong\CC$, as
follows from Serre duality.

We now return to $g=3$.
Assume first that $\mu=(a,b,c)$ is even, i.e., $a+b+c$ is even.
Then $\mu+(13,13,13)$ is odd and so are
the Teichm\"uller modular forms in the isomorphisms above. As we saw
in \S7, odd Teichm\"uller modular forms of genus $3$ are divisible
by $\chi_9$, and the even quotients are pullbacks of Siegel modular forms.
So for $g=3$ and $\mu$ even,
the isomorphisms above may be rewritten as follows:
$$ F^mH^{6}(\M{3},\VlaC)\cong 
F^mH^{6}(\Mct{3},\VlaC)\cong
M_{a-b,b-c,c+4}\,, \eqno(17) $$
$$ F^mH^{6}_c(\M{3},\VlaC) \cong F^mH^{6}_c(\Mct{3},\VlaC)
\cong 
S_{a-b,b-c,c+4}\,, \eqno(18) $$
with $m=a+b+c+6$. 
Therefore, we also have isomorphisms with $F^mH^6(\A{3},\VVl)$ 
resp.~$F^mH^6_c(\A{3},\VVl)$.

An interesting particular case of equation (17) %% \eqno
is obtained by taking $\mu=0$. We find that 
$F^6H^6(\M{3},\CC)$ and $F^6H^6(\Mct{3},\CC)$ are
onedimensional and naturally isomorphic to $M_{0,0,4}$,
which is spanned by the Eisenstein series $\alpha_4$
(see~\cite{Tsuyumine}). This is Looijenga's class
of type $(6,6)$ (cf.~\cite{Looijenga}).
(The dual class in $H^6_c$ is of type $(0,0)$ and
$F^6H^6_c=(0)$, since $\alpha_4$ is not a cusp form.)

Assume next that $\mu$ is odd, so $\nu=\mu+(13,13,13)$ is even.
Equations (14)--(16) % \eqno
say the following.
The classes in $F^{a+b+c+6}$ of the middle cohomology of $\VlaC$
on $\M{3}$ correspond to Teichm\"uller forms of type $\nu$ that
vanish along the entire boundary. Only those corresponding to forms
that vanish at
least twice along the boundary divisors parametrizing
reducible curves are restrictions from $\Mct{3}$.
Only the forms vanishing at least twice along the entire boundary
correspond to compactly supported classes.

For concrete examples, note that
Bergstr\"om~\cite{Jonas} has determined $e_c(\M{3},\VlaC)$, hence
$e(\M{3},\VlaC)$ for all $\mu$ with $|\mu|=a+b+c\le7$.
E.g., for $\mu=(1,1,1)$, so $m=9$, one has $e_c=-L^7-L^2+L+1$, so
$e=L^9+L^8-L^7-L^2$. Then $F^me=L^9$, but it is not yet clear
which cohomology groups contribute to this. Since the virtual
cohomological dimension of $\M{3}$ equals $7$ (cf.~\cite{Harer}),
only classes in $H^6$ and $H^7$ can contribute, and there must be
at least one class in $F^9H^6$. Now
$$F^9H^6(\M{3},\VlaC)\cong H^0(\barM{3},\mathcal{O}(14\lambda-\delta)).$$
According to~\cite{Tsuyumine}, the space $S_{0,0,14}$ is spanned
by the cusp form $\beta_{14}$. We conclude that $\beta_{14}$
vanishes along $\A{2,1}$ and that $F^9H^7=(0)$. In fact, by loc.~cit.,
$\beta_{14}$ vanishes twice along $\A{2,1}$, so the corresponding
class is a restriction from $\Mct{3}$.

In~\cite{BFvdG2}, Bergstr\"om and two of the present authors have
counted curves of genus $2$ and $3$ over finite fields and their
numbers of points and in this way determined the $\FF_q$-traces
of $e_c(\M{3},\VVl^{\prime})$ and $e_c(\A{3},\VVl)$ for $q\le25$.
As discussed in loc.~cit., this has led to a complete conjecture for 
these motivic Euler characteristics in the case of $\A{3}$.
In the case of $\M{3}$, we have obtained
precise conjectures for nearly all $\mu$ with $|\mu|\le20$ 
(for three $\mu$ with $|\mu|=19$, some information is missing).
The work of
Chenevier-Renard \cite{CheRen},
Chenevier-Lannes \cite{CheLan},
Ta\"ibi \cite{Taibi},
and M\'egarban\'e \cite{Meg}
has played a particularly important role here.
Below, we use some of the results to provide further examples
(in effect, we obtain further evidence for the conjectures).
Note that the `motivic' conjectures tell us in particular
where terms of the maximum possible Hodge degree 
are to be expected.

We first consider scalar-valued Teichm\"uller modular forms,
so $\mu=(k,k,k)$. Here are three more even cases:
\begin{enumerate}
\item $k=2$, $m=12$: $e_c=1$, $e=L^{12}$; this corresponds to the
Eisenstein series $\alpha_6$~;
\item $k=4$, $m=18$: $e_c=L^5+1$, $e=L^{18}+L^{13}$; $F^me$
corresponds to the Eisenstein series $\alpha_4^2$~;
\item $k=6$, $m=24$: $F^me=L^{24}+L^7S_{18}$ corresponds to the
space spanned by the modular forms $\alpha_4\alpha_6$ and $\alpha_{10}\,$.
\end{enumerate}
More interesting are two more odd cases:
\begin{enumerate}
\item $k=3$, $m=15$: $F^me=L^{15}$; indeed,
$H^0(\barM{3},\mathcal{O}(16\lambda-\delta))$ is spanned by $\beta_{16}$,
since no nontrivial linear combination of the other two generators
$\alpha_4\alpha_{12}$ and $\alpha_{16}-\tfrac{27}4\alpha_6\alpha_{10}$ 
of the space of cusp forms of weight $16$
vanishes along $\A{2,1}$~;
\item $k=5$, $m=21$: in this case, we conjecture that
$$e_c(\M{3},\VV_{5,5,5}^{\prime})=
L^7+2L^6+5L^5+6L^4+4L^3+3L^2+2L+1+S[4,10],$$
so $F^me= L^{21}+F^mS[4,10]\cong \CC^2$. 
On the other hand,  
$H^0(\barM{3},\mathcal{O}(18\lambda-\delta))$ is spanned by
$\alpha_4\beta_{14}$ and $\chi_{18}$, since
no linear combination of $\alpha_6\alpha_{12}$
and $\alpha_{18}-27\alpha_4^2\alpha_{10}$ (the other 
generators of the space of cusp forms
of weight $18$) vanishes along $\A{2,1}\,$.
Note also that $F^me_c=F^mS[4,10]\cong\CC$. This corresponds to
$H^0(\barM{3},\mathcal{O}(18\lambda-2\delta))=\CC\cdot\chi_{18}$. 
\end{enumerate}

\bigskip

More generally, for $\mu=(a,b,c)$ odd, it is natural to consider
$$
H^0(\barM{3},\EE^{\prime}_{\mu}\otimes \mathcal{O}(13\lambda-2\delta)),
$$
since this corresponds to compactly supported cohomology by
equation~(16) % \eqno
and thus to the data obtained by counting curves.

This space is a subspace of 
$$ S_{a-b,b-c,c+13}^2\,,$$
which by Theorem~\ref{the_isom} is isomorphic to
$$ C_{d,(a-5,b-5,c-5)}(-\tfrac{d-2}2DC),$$
where $d=a+b+c-15$ is even. In particular,
there are no contributions for $|\mu| < 15$.

More precisely, the subspace consists of those forms
vanishing to order at least~$2$ along $\A{2,1}$
(or along $\delta_1$ when considered as Teichm\"uller forms).

Now the isomorphism is given by $c\mapsto \gamma(c)\chi_{18}$
and we know that $\chi_{18}$ has order $6$ along $\A{2,1}$.
So whenever $\gamma(c)$ has order at least~$-4$ along
$\delta_1$, we obtain an element of the subspace
of interest. Since $\chi_{4,0,-1}$ has a simple pole
along $\delta_1$ (by Lemma~\ref{I2notI3} and Proposition~\ref{basic}), this is
automatically the case for $d\le4$.

The case $d=0$, hence $\mu=(5,5,5)$ has been discussed
already: $m=21$ and 
$$F^me_c(\VV_{5,5,5}^{\prime}\otimes\CC) =F^mS[4,10]$$
corresponds to the space spanned by $\chi_{18}\,$.

In the case $d=2$, there are three nonzero spaces
of concomitants (\S\ref{cased=2}), each of dimension~$1$.
The corresponding $\mu$ are $(11,3,3)$, $(9,5,3)$, and $(7,7,3)$.
The associated (irreducible) `motives' are
$M[23,13,5]$, $S[8,9]$, and $M[23,15,3]$, respectively.
Here, $S[8,9]$ is the motive associated to Siegel cusp forms
of degree~$2$ (!) of type $\Sym^8\det^9$. The other two motives
are {\sl not} associated to Siegel cusp forms. They are
$6$-dimensional motives of weight $23$, first identified
by Chenevier-Renard \cite{CheRen}; later, M\'egarban\'e \cite{Meg}
computed their $\FF_q$-traces for many small~$q$, in complete
agreement with the counting data of \cite{BFvdG2}. In denoting
these motives, we essentially follow the notation of 
\cite{CheRen} and \cite{Meg}; e.g., the Hodge degrees of
$M[23,13,5]$ are
$$ 0,\quad 5,\quad 9,\quad 14,\quad 18,\quad 23 $$
(so the successive widths are $23$, $13$, and $5$).

\bigskip

We have explicitly computed all (spaces of) concomitants of ternary
quartics of degree at most~$6$. 
We have also determined the subspaces of concomitants vanishing
to order at least $r$ along $DC$, for each $r\ge0$.
(The computations are
quite involved and we will discuss them in a future paper.)

In particular, for $d=4$, we find nonzero concomitants vanishing
along $DC$ exactly in the $13$ following cases:
$$
\mu=(14,4,1),\quad (13,5,1),\quad (12,5,2),\quad (12,4,3),\quad
(11,7,1),\quad (11,6,2),\quad (11,5,3),$$
$$(10,7,2),\quad (10,6,3),\quad(10,5,4), \quad (9,7,3),\quad (9,5,5), \quad
(8,7,4).$$
The corresponding motives have been identified in $10$ of the $13$ cases.
Their dimensions are $4$, $6$, or $8$. The work of M\'egarban\'e \cite{Meg}
has played a crucial role here. There are reasons to believe
that in the $3$ remaining cases the motives will be $12$-dimensional; in any
case, these motives have not yet been identified.

The conjectural formulas for the Euler characteristics
$e_c(\M3,\Vla)$ with $\mu$ odd are rather complicated and
contain many more terms than just the motives mentioned
above. 
The techniques and computations that were used to obtain
the formulas, and the formulas themselves, will be discussed elsewhere.

\end{section}
%%%%%%%%%%%%%%%
\begin{section}{Appendix: Teichm\"uller modular forms extend 
to $\barM{g}$ for $g\ge3$} \label{appendix}

The result below was found 
by Farkas, Pandharipande, and the second author, after
a talk in Berlin (January 2019), in which 
Proposition~\ref{extensiontoM3bar} and its proof were presented.
The visit to Berlin was supported by the Einstein Stiftung.

\bigskip
Let $\M{g}$ denote the moduli space of curves of genus $g\ge2$.
The Torelli morphism $t\colon\M{g}\to \A{g}$
is a morphism of algebraic stacks.
%of degree $2$ ramified along the hyperelliptic locus.
By pullback under $t$
 we obtain the Hodge bundle ${\EE}^{\prime}$
on $\M{g}$ and for each irreducible representation $\rho$ of ${\rm GL}(g)$
a vector bundle ${\EE}^{\prime}_{\rho}$ on $\M{g}\,$. Sections of such a bundle
${\EE}^{\prime}_{\rho}$ are called Teichm\"uller modular forms of degree~$g$.
The vector bundle ${\EE}^{\prime}$ and hence all the
${\EE}^{\prime}_{\rho}$ extend in a natural way over the
Deligne-Mumford compactification
$\overline{\mathcal M}_g\,$.
The purpose of this note is to
show that a holomorphic section of
${\EE}^{\prime}_{\rho}$ automatically extends to a holomorphic section of
the extended bundle, for $g\ge3$. (As Ichikawa already observed,
the result fails for $g=2$: Igusa's cusp form $\chi_{10}$ vanishes
exactly along $\Delta_0$ and (doubly) along $\Delta_1\,$, so its inverse
is a regular section of $\det^{-10}({\EE}^{\prime})$ on $\M2$ which
does \emph{not} extend to $\barM2\,$.)

\begin{proposition}
Let $g\ge3$ be an integer.
Let $\rho$ be an irreducible representation of ${\rm GL}(g)$
and let ${\EE}^{\prime}_{\rho}$ be the bundle on $\barM{g}$
arising from the Hodge bundle ${\EE}^{\prime}$ of rank $g$
by applying the Schur functor corresponding to $\rho$.
Then a section $s$ of ${\EE}^{\prime}_{\rho}$ over $\M{g}$
extends to a regular section of ${\EE}^{\prime}_{\rho}$
over $\barM{g}\,$.
\end{proposition}

\begin{proof}
We first show that $s$ extends over the boundary divisors
parametrizing reducible curves.
Let $\Delta_i$ with $0<i\le g/2$ be such a divisor. Let
$D$ be the open boundary stratum of $\Delta_i\,$.
So $D\cong\M{i,1}\times\M{g-i,1}$ or its quotient
by $S_2$ when $i=g/2$. Let $X=\M{g}\cup D$. We show
that the global sections of ${\EE}^{\prime}_{\rho}$
on $\M{g}$ agree with those on $X$, as follows.

On $X$, we have the following exact sequences:
$$0\to{\mathcal O}(-D)\to{\mathcal O}\to{\mathcal O}_D\to 0,$$
$$0\to{\mathcal O}\to{\mathcal O}(D)\to{\mathcal O}(D)_D\to 0,$$
$$0\to{\EE}^{\prime}_{\rho}\to{\EE}^{\prime}_{\rho}(D)
\to({\EE}^{\prime}_{\rho}(D))|_D\to 0.$$
We want
$$H^0(X,{\EE}^{\prime}_{\rho})=H^0(X,{\EE}^{\prime}_{\rho}(D))$$
and obtain this by showing that
$$H^0(X,({\EE}^{\prime}_{\rho}(D))|_D)=H^0(D,{\EE}^{\prime}_{\rho}(D))=0.$$
For the latter vanishing, consider the complete curve $B$ in $D$
obtained by varying the attachment point on a given smooth curve
$C$ of genus $g-i$ (since $g\ge3$ we have $g-i\ge2$). On such a curve,
${\EE}^{\prime}$ and thus ${\EE}^{\prime}_{\rho}$ are constant, so a direct sum
of trivial line bundles. The line bundle ${\mathcal O}(D)|_B$
is $-K_C\,$, of strictly negative degree, so its global sections on
$B$ are zero. The same holds then for the global sections
of ${\EE}^{\prime}_{\rho}(D)$ on $B$ and then also on $D$, since the curves
$B$ fill out $D$.
Thus
$$H^0(X,{\EE}^{\prime}_{\rho})=H^0(X,{\EE}^{\prime}_{\rho}(D))$$
as desired. By twisting more with ${\mathcal O}(D)$ we find
similarly that
$$H^0(X,{\EE}^{\prime}_{\rho})=H^0(X,{\EE}^{\prime}_{\rho}(kD))$$
for an arbitrary positive integer $k$. Thus a global section
of ${\EE}^{\prime}_{\rho}$ on $\M{g}\,$, which a priori could have
a pole along $D$ of some order, is in fact regular along $D$
and extends to $X$.

Next, we consider $\Delta_0\,$. Let $D$ be the open part consisting
of the open boundary stratum and the stratum corresponding to
smooth curves of genus $g-1$ with a singular elliptic tail attached.
So $D\cong ({\mathcal C}_{g-1}^2) /S_2\,$. Again, let $X=\M{g}\cup D$.
We proceed as above, but for the argument on $D$ we may as well
work on the double cover $D_2\cong {\mathcal C}_{g-1}^2\,$.
Let $B$ in $D$ be the complete curve obtained
by varying the (second) point $Q$ on a given one-pointed smooth curve $(C,P)$
of genus $g-1$. (When $Q=P$, we get a stable curve in the smaller
stratum, of course.) Now ${\mathcal O}(D)|_B$ is $-K_C(2P)$, again
of negative degree.
On $D_2\,$, the Hodge bundle ${\EE}^{\prime}$ sits in an exact sequence
$$ 0\to {\FF}^{\prime} \to {\EE}^{\prime}
\,{\buildrel r \over \longrightarrow}\,
 {\mathcal O}_{D_2}\to 0,$$
%% \atop
where ${\FF}^{\prime}$ is the Hodge bundle pulled back from
genus $g-1$ and the map $r$ to ${\mathcal O}_{D_2}$ is obtained by taking
the residue at~$P$. The restriction of ${\FF}^{\prime}$ to $B$
is trivial. A global section $t$ of ${\EE}^{\prime}(D)|_{D_2}$ on $B$ gives
a global section of ${\mathcal O(D)}_{D_2}\,$ on $B$, which must be trivial,
so $t$ comes from a section of ${\FF}^{\prime}(D)|_{D_2}$ on $B$, which
must be trivial as well.
To show that a global
section of ${\EE}^{\prime}_{\rho}(D)$ on $B$ also must vanish,
we argue as follows.
Firstly, global sections of $(({\EE}^{\prime})^{\otimes n})(D)$ on $B$
vanish for every $n$; e.g., for $n=3$ a global section $t$ gives via
$r\otimes r\otimes r$ a global section of ${\mathcal O}(D)$,
which vanishes;
so $t$ comes from a section of the subbundle
$$({\FF}^{\prime}\otimes {\EE}^{\prime} \otimes {\EE}^{\prime}
+ {\EE}^{\prime}\otimes {\FF}^{\prime} \otimes {\EE}^{\prime}
+ {\EE}^{\prime}\otimes {\EE}^{\prime} \otimes {\FF}^{\prime})(D)$$
of rank $g^3-1$; applying
$r\otimes r\otimes 1$, $r\otimes 1\otimes r$, and $1\otimes r\otimes r$,
we find $t$ must come from a section of
$$({\FF}^{\prime}\otimes {\FF}^{\prime} \otimes {\EE}^{\prime}
+ {\FF}^{\prime}\otimes {\EE}^{\prime} \otimes {\FF}^{\prime}
+ {\EE}^{\prime}\otimes {\FF}^{\prime} \otimes {\FF}^{\prime})(D);$$
applying
$r\otimes 1\otimes 1$, $1\otimes r\otimes 1$, and $1\otimes 1\otimes r$,
we find $t$ comes from a section of
$$({\FF}^{\prime}\otimes {\FF}^{\prime} \otimes {\FF}^{\prime})(D),$$
which must vanish.
Next, on $D_2$ we have that $\det({\EE}^{\prime})$ is trivial. So with
$\rho=(\rho_1\ge\rho_2\ge\dots\ge\rho_g)$, we may and will assume
that $\rho_g=0$. Finally, ${\EE}^{\prime}_{\rho}$ occurs in the
decomposition of $({\EE}^{\prime})^{\otimes n}$ as a summand, with
$n=|\rho|=\rho_1+\dots+\rho_{g-1}$. More precisely, the bundle
of ${\rm GL}(g)$-equivariant homomorphisms from
${\EE}^{\prime}_{\rho}$ to $({\EE}^{\prime})^{\otimes n}$ is free
of rank equal to the multiplicity of the former
in the decomposition of the latter. Choose a nonzero section $u$ of this
bundle; via $u$, a section of ${\EE}^{\prime}_{\rho}$ gives a section of
$({\EE}^{\prime})^{\otimes n}$, by taking the zero section of the
other isotypical components. It follows that a global section of
${\EE}^{\prime}_{\rho}(D)$ on $B$ necessarily vanishes, as claimed.
The same holds then for the global sections on $D$ and then also
for ${\EE}^{\prime}_{\rho}(kD)$ for positive~$k$.

We conclude that a section of ${\EE}^{\prime}_{\rho}$ over $\M{g}$
extends over nonempty open parts of each boundary divisor. This is enough,
since the complement is of codimension~$2$ and $\barM{g}$ is a smooth stack.
\end{proof}
\end{section}

%%%%%%%%%%%%%%%%%%%%%%%%%%%%%%%%%%%%%%%%

\end{document}